\documentclass[11pt]{amsart}
\usepackage{amscd}
\usepackage{amsfonts}
\usepackage{amsmath,amsthm,hyperref}
\usepackage[all]{xy}
\usepackage{amsmath,amsthm,hyperref}
\usepackage{amsmath,amssymb,amsthm,latexsym}
\usepackage{amscd}
\usepackage{mathdots}
\newtheorem{theorem}{Theorem}[section]
\newtheorem{proposition}[theorem]{Proposition}
\newtheorem{definition}[theorem]{Definition}

\newtheorem{lemma}[theorem]{Lemma}

\theoremstyle{remark}
\newtheorem{remark}[theorem]{Remark}

\usepackage[dvips]{graphicx}
\usepackage{abstract}
\usepackage[labelfont=bf]{caption}
\usepackage{amscd}
\usepackage{comment}
\usepackage{amsfonts}
\usepackage{color}
\usepackage{multicol}
\usepackage[all]{xy}
\usepackage{amsmath,amssymb,amsthm,latexsym}
\usepackage{amscd}

\makeatletter
\renewcommand*{\@cite@ofmt}{\bfseries\hbox}
\g@addto@macro\bfseries{\boldmath}
\makeatother

\numberwithin{equation}{section}

\newcommand{\real}{\mathbb R}
\newcommand{\C}{\mathbb C}
\newcommand{\bdm}{\begin{displaymath}}
\newcommand{\edm}{\end{displaymath}}
\newcommand{\beq}{\begin{equation}}
\newcommand{\beqa}{\begin{eqnarray}}
\newcommand{\beqas}{\begin{eqnarray*}}
\newcommand{\eeq}{\end{equation}}
\newcommand{\eeqa}{\end{eqnarray}}
\newcommand{\eeqas}{\end{eqnarray*}}
\newcommand{\dd}{\textup{d}}
\newcommand{\bbar}{\left( \begin{array}}
\newcommand{\ebar}{\end{array} \right)}
\newcommand{\la}{\langle}
\newcommand{\ra}{\rangle}

\begin{document}
\title[Construction of Willmore two-spheres]{\bf{Construction of Willmore two-spheres via harmonic maps into $SO^+(1,n+3)/(SO^+(1,1)\times SO(n+2))$}}
\author{Peng Wang }
\date{}
\maketitle

\begin{center}
{\bf Abstract}
\end{center}

This paper aims to provide a description of totally isotropic Willmore two-spheres and their adjoint transforms.
We first recall the isotropic harmonic maps which are introduced by H\'{e}lein, Xia-Shen and Ma for the study of Willmore surfaces.
Then we derive a  description of the normalized potential (some Lie algebra valued meromorphic 1-forms) of totally isotropic Willmore two-spheres in terms of  the isotropic harmonic maps. In particular, the corresponding isotropic harmonic maps are of finite uniton type. The proof also contains a concrete way to construct examples of totally isotropic Willmore two-spheres and their adjoint transforms. As illustrations, two kinds of examples are obtained this way. \\

{\bf Keywords:}  Willmore surfaces;  Isotropic Willmore two-spheres; DPW method; adjoint transform; isotropic harmonic maps.\\

MSC(2010): 53A30; 58E20; 53C43; 53C35

\tableofcontents
\section{Introduction}

Totally isotropic surfaces was first introduced by Calabi \cite{Calabi} in the study of  the global geometry of minimal two-spheres.
In the study of Willmore two-spheres, totally isotropic surfaces also play an important role \cite{Ejiri1988,Mon,Mus1}. Recently, Dorfmeister and Wang used the DPW method for the conformal Gauss map to study Willmore surfaces \cite{DoWa1}. They obtained the first new Willmore two-sphere in $S^6$, which admits no dual surface.
Along this way, Wang \cite{Wang-iso} gives a description of all totally isotropic Willmore two-spheres in $S^6$, in terms of the normalized potentials of their conformal Gauss maps.

Although Willmore two-spheres may have no dual surfaces, they do admit another kind of transforms, i.e., adjoint transforms introduced by Ma \cite{Ma2006}. The main idea of the adjoint transforms is to find out another Willmore surface located in the mean curvature spheres of the original one and having the same complex coordinates. A somewhat surprising result derived by Ma states that a Willmore surface and its adjoint surface in $S^{n+2}$ also provide  a harmonic map into the Grassmannian $Gr_{1,1}(\mathbb R^{n+3}_1)$, which was first discovered by H\'{e}lein \cite{Helein,Helein2} and generalized by Xia and Shen \cite{Xia-Shen}.
This kind  of harmonic maps also appeared naturally when Brander and Wang considered the Bj\"{o}rling problems of Willmore surfaces \cite{BW}.

 However, such harmonic maps will have singularities in general (See Section 6 for example). So it is very hard to use them to discuss the global geometry of Willmore surfaces. Due to this reason, in \cite{DoWa1}  Dorfmeister and Wang mainly dealt with the conformal Gauss maps of Willmore surfaces, which are anther kind of harmonic maps related to Willmore surfaces globally.

Although it exists locally in general, the harmonic map given by  a Willmore surface and its adjoint surface is very simple and provides the Willmore surface and its adjoint surface immediately. So it is natural to use this harmonic map to describe totally isotropic Willmore two-spheres. In particular, this provides a more simple way to derive examples of totally isotropic Willmore two-spheres and their adjoint surfaces, in contrast to  using the conformal Gauss maps of Willmore surfaces \cite{Wang-iso}.

In this paper, we will give a characterization of the harmonic map given by a totally isotropic Willmore two-sphere and its adjoint surface.
In particular, such harmonic maps are very simple so that one can obtain a concrete algorithm to construct all of them, which is the main topic of this paper. As illustrations, we also derive two kinds of examples.

The main idea of our work are based on the DPW method  for harmonic maps \cite{DPW, Helein} and the description of harmonic maps of finite uniton type \cite{BuGu,Gu2002,DoWa2}. The DPW method \cite{DPW} gives a way to produce harmonic maps in terms some meromorphic 1-forms, i.e., normalized potentials. The work of \cite{BuGu,Gu2002,DoWa2} states that harmonic maps of finite uniton type can be derived in a more convenient way, that is, the normalized potentials must take values in some nilpotent Lie algebra. This permits a way to derive such harmonic maps in an explicit way. On the other hand, due to \cite{Uh,BuGu,Gu2002,DoWa2}, harmonic maps from two-spheres into an inner symmetric space will always be of finite uniton type.
This provides a way to classify all Willmore two-spheres in terms of their conformal Gauss maps \cite{Wang-1}. For the harmonic maps used in this paper, a main problem is that they are not globally well-defined in general. So one can not apply the theory to such harmonic maps. But using Wu's formula and the description of the normalized potentials of harmonic maps of finite uniton type, we are able to show that the harmonic map given by a totally isotropic Willmore two-sphere and its adjoint surface is also of finite uniton type. So far we do not have a clear explain for this phenomena, which may need a detailed discussion on the Iwasawa cells of the corresponding non-compact Loop groups. We hope to continue this study in future publication.\\

This paper is organized as follows:
In Section 2, we first recall some basic results about Willmore surfaces and their adjoint transforms.
 Then in Section 3 we discuss the isotropic harmonic maps given by  Willmore surfaces and their adjoint transforms.
 Section 4 provides a description of the normalized potentials of totally isotropic Willmore two-spheres in terms of  the isotropic harmonic maps. The converse part, i.e., that generically such normalized potentials will always produce totally isotropic Willmore surfaces and their adjoint transforms, is the main content of Section 5. Using these results, we also derive some concrete examples in Section 5. Then we end the paper by
 Section 6, which contain the technical and tedious computations of Section 5.

\section{Willmore surfaces and adjoint surfaces}  \label{section2}

In this  section we will first recall the basic surface theory of Willmore surfaces in $ S^{n+2}$ in the spirit of the treatment of \cite{BPP,Ma}.
Then we will collect the descriptions of the  adjoint transforms of Willmore surfaces \cite{Ma2006}.
We refer to \cite{Ma,Ma2006,BW} for more details.

\subsection{Review of Willmore surfaces in $ S^{n+2}$}

Let $\mathbb{R}^{n+4}_1$  be the Minkowski space, with a Lorentzian metric $\langle x,y\rangle =-x_{1}y_1+\sum_{j=2}^{n+3}x_jy_j=x^t I_{1,n+3} y,$ $ I_{1,n+3}=\textup{diag}(-1,1,\cdots,1).$
Let $\mathcal{C}_+^{n+3}:= \{ x \in \mathbb{R}^{n+4}_{1} |\langle x,x\rangle=0, x_1>0 \} $ be the forward light cone.
Let $
Q^{n+2}:=\{\ [x]\in\mathbb{R}P^{n+3}\ |\ x\in \mathcal{C}_+^{n+3}\}$ be the the projective light cone with the induced conformal metric.
Then $Q^{n+2}$ is conformally equivalent to $S^{n+2}$, and the conformal group of
$S^{n+2}\cong Q^{n+2}$ is  the orthogonal group $O(1,n+3)/\{\pm1\}$ of
$\mathbb{R}^{n+4}_1$, acting on $Q^{n+2}$ by
$
T([x])=[Tx]$ for any $T\in O(1,n+3).
$
Let $SO^+(1,n+3)$ be the connected component of $O(1,n+3)$ containing $I$, i.e.,
\[SO^+(1,n+3)=\{T\in O(1,n+3)|\ \det T=1,\hbox{ $T$ preserves the time direction of } \mathbb{R}^{n+4}_1\}.\]

Let $y:M \rightarrow  S^{n+2}$ be a conformal immersion from a Riemann surface $M$, with $z$ a local complex coordinate on $U\subset M$ and $\la y_z,y_{\bar z}\ra=\frac{1}{2}e^{2\omega}$. The lift $Y:U\rightarrow \mathcal{C}_+^{n+3}$ is called a canonical lift of $y$ with respect
to $z$, satisfying
$|{\rm d}Y|^2=|{\rm d}z|^2$.
 Then there is a bundle decomposition
\[
M\times
\mathbb{R}^{n+4}_{1}=V\oplus V^{\perp}, \hbox{ with } V={\rm Span}_{\mathbb R}\{Y,{\rm Re}Y_{z},{\rm Im}Y_{z},Y_{z\bar{z}}\},\ V^{\perp}\perp V.
\]
Here $V$ is a Lorentzian rank-4 sub-bundle. This decomposition is independent of the choice of $Y$
and $z$. We denote by $V_{\mathbb{C}}$ and
$V^{\perp}_{\mathbb{C}}$ as  their
complexifications.
There exists a unique section $N\in\Gamma(V)$ such that $
\langle N,Y_{z}\rangle=\langle N,Y_{\bar{z}}\rangle=\langle
N,N\rangle=0,\langle N,Y\rangle=-1.$
Noting that $Y_{zz}$ is orthogonal to
$Y$, $Y_{z}$ and $Y_{\bar{z}}$, there exists a complex function $s$
and a section $\kappa\in \Gamma(V_{\mathbb{C}}^{\perp})$ such that
$
Y_{zz}=-\frac{s}{2}Y+\kappa.
$
This defines two basic invariants $\kappa$ and $s$ depending on
coordinates $z$, \emph{the conformal Hopf differential} and
\emph{the Schwarzian} of $y$  \cite{BPP}. Let $D$
denote the normal connection and $\psi\in
\Gamma(V_{\mathbb{C}}^{\perp})$ any section of the normal bundle.
The structure equations can be given as follows:
\begin{equation}\label{eq-stru}
\left\{\begin {split}
Y_{zz}&=-\frac{s}{2}Y+\kappa,\\
Y_{z\bar{z}}&=-\langle \kappa,\bar\kappa\rangle Y+\frac{1}{2}N,\\
N_{z}&=-2\langle \kappa,\bar\kappa\rangle Y_{z}-sY_{\bar{z}}+2D_{\bar{z}}\kappa,\\
\psi_{z}&=D_{z}\psi+2\langle \psi,D_{\bar{z}}\kappa\rangle Y-2\langle
\psi,\kappa\rangle Y_{\bar{z}}.
\end {split}\right.
\end{equation}
The conformal Gauss, Codazzi and Ricci equations as integrable
conditions are:
\begin{equation}\label{eq-integ}
\frac{s_{\bar{z}}}{2}=3\langle
\kappa,D_z\bar\kappa\rangle +\langle D_z\kappa,\bar\kappa\rangle,\
{\rm Im}(D_{\bar{z}}D_{\bar{z}}\kappa+\frac{\bar{s}}{2}\kappa)=0,\
 D_{\bar{z}}D_{z}\psi-D_{z}D_{\bar{z}}\psi =
2\langle \psi,\kappa\rangle\bar{\kappa}- 2\langle
\psi,\bar{\kappa}\rangle\kappa.
\end{equation}

 \emph{The conformal Gauss map} of $y$ is defined as follow.
\begin{definition} \cite{Bryant1984,BPP,Ejiri1988,Ma}
For a conformally immersed surface $y:M\to  S^{n+2}$, \emph{the conformal Gauss map $Gr(p):  M\rightarrow
Gr_{3,1}(\mathbb{R}^{n+4}_{1})=SO^+(1,n+3)/(SO^+(1,3)\times SO(n))
$} of
$y$ is defined as
\[
Gr(p):=V_{p}.
\]
So locally we have
$Gr=Y\wedge Y_{u}\wedge Y_{v}\wedge N=-2i\cdot Y\wedge Y_{z}\wedge
Y_{\bar{z}} \wedge N,$ with $z=u+iv.$

\end{definition}

 Direct computation shows that $Gr$ induces a conformal-invariant
metric
$
\mathbf{g}:=\frac{1}{4}\langle {\rm d}Gr,{\rm d}Gr\rangle=\langle
\kappa,\bar{\kappa}\rangle|\dd z|^{2}
$
on M. Note $\mathbf{g}$ degenerates at umibilic points of $y$.  The Willmore
functional and Willmore surfaces can be defined by use of this metric.

\begin{definition} \emph{The Willmore functional} of $y$ is
defined as:
\[
W(y):=2i\int_{M}\langle \kappa,\bar{\kappa}\rangle \dd z\wedge
\dd \bar{z}.
\]
An immersed surface $y:M\rightarrow  S^{n+2}$ is called a
\emph{Willmore surface} if it is a critical surface of the Willmore
functional with respect to any variation of the map $y:M\rightarrow
  S^{n+2}$.
\end{definition}
It is well-known that \cite{Bryant1984,BPP,Ejiri1988,Wang1998}  $y$ is Willmore if and only if
 \begin{equation}\label{eq-Willmore}
D_{\bar{z}}D_{\bar{z}}\kappa+\frac{\bar{s}}{2}\kappa=0;
 \end{equation}
if and only if
 the conformal Gauss map $Gr:M\rightarrow Gr_{3,1}(\mathbb{R}^{n+3}_{1})$ is  harmonic. We refer to \cite{DoWa1} for
 the conformal Gauss map approach for Willmore surface.

\subsection{Adjoint transforms of a Willmore surface and the second harmonic map related to Willmore surfaces}

Transforms play an important role in the study of Willmore surfaces. For a Willmore surface $y$ in $S^3$, it was shown by Bryant in the seminal paper \cite{Bryant1984} that they always admit a unique dual surface which may have
branch points or degenerate to a point. Hence the dual surface is either degenerate or has the same complex coordinate and the same conformal Gauss map as $y$ at the points it is immersed. This duality theorem, however, does not hold in general when the codimension is bigger than $1$ (\cite{Ejiri1988}, \cite{BPP}, \cite{Ma2006}). To characterize  Willmore surfaces with dual surfaces, in \cite{Ejiri1988} Ejiri  introduced the notion of {\em S-Willmore surfaces}. Here we define it slightly differently to include all Willmore surfaces with dual surfaces:
\begin{definition}
A Willmore immersion $y:M^2\rightarrow  S^{n+2}$ is called an S-Willlmore surface if its conformal Hopf differential satisfies
\[
D_{\bar{z}}\kappa || \kappa,
\]
i.e. there exists some function $\mu$ on $M$ such that $D_{\bar{z}}\kappa+\frac{\mu}{2}\kappa=0$.
\end{definition}
A basic result of \cite{Ejiri1988} states that a Willmore surface admits a dual surface if and only if it is S-Willmore. Moreover the dual surface is also Willmore at the points it is immersed.

To consider the generic Willmore surfaces, Ma introduced the adjoint transform of a Willmore surface $y$ \cite{Ma2006,Ma}.
An adjoint transform of $y$ is a conformal map $\hat y$ which is located on the mean curvature sphere of $y$ and satisfies some additional condition.
To be concrete we have the following
\subsubsection{Adjoint transforms}
Let $y:U\rightarrow  S^{n+2}$ be an umbilic free Willmore surface  with canonical lift $Y$ with respect to $z$ as above. Set
\begin{equation}\label{eq-hat-Y-def}
\hat{Y}=N+ \bar\mu Y_{z}+ \mu Y_{\bar{z}}+ \frac{1}{2}|\mu|^2Y,
\end{equation}
with $\mu \dd z=2\langle\hat{Y},Y_z\rangle \dd z$ a  connection 1--form.
Direct computation yields \cite{Ma2006}
\begin{equation}\label{eq-hat-Y-z}
\hat{Y}_{z}=\frac{\mu}{2} \hat{Y}+\theta \left(Y_{\bar{z}}+\frac{\bar\mu}{2}Y\right)+\rho  \left(Y_{z}+\frac{\mu}{2}Y\right)+2\zeta
\end{equation}
with
\[
\theta:=
\mu_z- \frac{\mu^2}{2}-s,\ \rho:=\bar\mu_{z}-2\langle \kappa,\bar\kappa \rangle,\ \zeta:=D_{\bar{z}}\kappa+\frac{\bar{\mu}}{2}\kappa.
\]
Now we define the adjoint surface as follow.
\begin{definition} \cite{Ma2006} The map   $ \hat{Y}: U \rightarrow S^{n+2}$  is called an adjoint transform of the Willmore surface $Y$ if the following two equations hold for $\mu$:
\begin{equation}\label{eq-h-1II}
\mu_z-\frac{\mu^2}{2}-s=0,  \hbox{\ Riccati equation,}
\end{equation}
\begin{equation}\label{eq-h-m}
\langle D_{\bar z}\kappa+\frac{\bar\mu}{2}\kappa, D_{\bar z}\kappa+\frac{\bar\mu}{2}\kappa \rangle =0.
\end{equation}
\end{definition}
Note that
$\hat{Y}$ is the dual surface of $Y$ if and only if $D_{\bar{z}}\kappa+\frac{\bar{\mu}}{2}\kappa=0$ (\cite{Bryant1984}, \cite{Ejiri1988},  \cite{Ma2006}).

\begin{theorem}\cite{Ma2006}\ {\em Willmore property and existence of adjoint transform:}
  The adjoint transform $\hat{Y}$ of a Willmore surface $y$ is also a Willmore surface (may degenerate). Moreover,
\begin{enumerate}
\item  If $\langle \kappa,\kappa\rangle \equiv0$, any solution to the equation \eqref{eq-h-1II} is also a solution to the equation
 \eqref{eq-h-m}. Hence, there exist infinitely many adjoint surfaces of $y$ in this case.
\item If $\langle \kappa,\kappa\rangle \neq 0$ and
$\Omega \dd z^6:=\langle D_{\bar{z}} \kappa,\kappa\rangle^2 -\langle \kappa,\kappa\rangle \langle D_{\bar{z}}\kappa,D_{\bar{z}}\kappa\rangle \dd z^6\neq 0,$
 there are exactly two different solutions to equation \eqref{eq-h-m}, which also solve \eqref{eq-h-1II}. Hence, there exist exactly two adjoint surfaces of $y$ in this case.
\item If $\langle \kappa,\kappa\rangle \neq 0$ and $\langle D_{\bar{z}} \kappa,\kappa\rangle^2 -\langle \kappa,\kappa\rangle \langle D_{\bar{z}}\kappa,D_{\bar{z}}\kappa\rangle \equiv 0,$ there exists a unique solution to  \eqref{eq-h-m}, which also solves  \eqref{eq-h-1II}.
     Hence, there exists a unique adjoint surface  of $y$ in this case.
\end{enumerate}
\end{theorem}
\begin{remark} In \cite{bq}, dressing transformations of constrained Willmore
surfaces are discussed in details. It stays unclear whether the adjoint transforms can be derived as a special kind of  dressing transformations.
\end{remark}

\subsubsection{Harmonic maps into $SO^+(1,n+3)/(SO^+(1,1)\times SO(n+2))$ related to
Willmore surfaces}

A crucial observation by H\'{e}lein etc. \cite{Helein,Helein2,Xia-Shen,Ma2006} is that $Y$ and $\hat Y$ produce furthermore a second useful harmonic map related to a Willmore surface $y$.
\begin{theorem}
Let  $[Y]$ be a Willmore surface. Let $\mu$ be a solution to the Riccati equation \eqref{eq-h-1II} on $U$, defining $\hat Y$
as \eqref{eq-hat-Y-def}.
Let
$\mathcal{F}_h:  U \to  SO^{+}(1,n+3)/(SO^+(1,1)\times SO(n+2))$ be
the map taking $p$ to  $Y(p)\wedge \hat{Y}(p)$. We have the following results.
\begin{enumerate}
\item (\cite{Helein,Helein2,Xia-Shen}) The map $ \mathcal{F}_h$ is harmonic,
and is called a {\em half--isotropic harmonic map with respect to $Y$}.
\item (\cite{Ma2006}) If $\mu$ also solves \eqref{eq-h-m}, i.e., $\hat Y$ is an adjoint transform of $y$,  then $\mathcal{F}_h$ is  \emph{conformally} harmonic, and is called an \emph{isotropic harmonic
map with respect to $Y$}.
\end{enumerate}
\end{theorem}

At umbilic points it is possible that there exists a limit of $\mu$ such that \eqref{eq-h-1II} holds. Due to the following lemma, the harmonic map $\mathcal{F}_h$ has no definition when $\mu$ tends to $\infty$.

\begin{lemma} \label{umbiliclemma}
\cite{Ejiri1988,DoWa1} At the umbilic points of $Y$,  the limit of $\mu$ goes to a finite number or infinity. When $\mu$ goes to infinity, $[\hat{Y}]$ tends to $[Y]$, and at the limit point
we have $[\hat{Y}]=[Y]$.
\end{lemma}

Restricting to the isotropic harmonic map, we have the following description.
\begin{theorem}\cite{Ma2006}, \cite{Helein2}, \cite{BW}
Let  $\mathcal{F}_h=Y\wedge\hat{Y}$ be an isotropic harmonic map.
Set $e_1, e_2\in\Gamma(V)$ with
 $Y_{z}+\frac{\mu}{2} Y=\frac{1}{2}(e_1-ie_2)$. Let
$\{\psi_j,j=1,\cdots,n\}$ be a frame of the normal bundle $V^{\perp}$. Assume that $\kappa=\sum_{j=1}^{n}k_j\psi_j,\ \zeta=\sum_{j=1}^{n}\gamma_j\psi_j,\  D_z\psi_j=\sum_{l=1}^{n}b_{jl}\psi_l,\  b_{jl}+b_{lj}=0.
$
Set
\[
F=\left(\frac{1}{\sqrt{2}}(Y+\hat{Y}),\frac{1}{\sqrt{2}}(-Y+\hat{Y}),e_1,e_2,\psi_1,\cdots,\psi_n\right).
\]
Then the Maurer-Cartan form
 $\alpha=F^{-1} \dd F=\alpha^\prime+\alpha^{\prime \prime}$ of $F$ has the structure:
\begin{equation}\label{eq-hm-m-c-h}
\alpha^\prime =\left(      \begin{array}{cc}
                     A_1 & B_1 \\
                     -B_1^tI_{1,1} & A_2 \\
                   \end{array}
                 \right)\dd z,
\end{equation}
with
\[
A_1=\left(
                      \begin{array}{cc}
                        0 & \frac{\mu}{2} \\
                        \frac{\mu}{2} & 0 \\
                      \end{array}
                    \right),\
 B_1=\left(
      \begin{array}{ccccccc}
                       \frac{1+\rho}{2\sqrt{2}} &  \frac{-i-i\rho}{2\sqrt{2}}  & \sqrt{2}\gamma_1 & \cdots & \sqrt{2}\gamma_n\\
                        \frac{1-\rho}{2\sqrt{2}} &  \frac{-i+i\rho}{2\sqrt{2}} &  -\sqrt{2}\gamma_1 & \cdots & -\sqrt{2}\gamma_n\\
      \end{array}
    \right)=\left(
              \begin{array}{c}
                b_1^t \\
                b_2^t \\
              \end{array}
            \right),
\]
and
 \begin{equation}\label{eq-B1-Ma-h}
 B_1B_1^t=0.
\end{equation}

Conversely, if $ \mathcal{F}=Y\wedge\hat Y:U \to  SO^{+}(1,n+3)/(SO^+(1,1)\times SO(n+2))$ is a conformal harmonic map satisfying \eqref{eq-B1-Ma-h}, then $ \mathcal{F}$ is an isotropic harmonic map and $Y$ and $\hat{Y}$  form a pair of adjoint Willmore surfaces at the points they are immersed. Moreover, set
\[
B_1=(b_1\ b_2)^t \hbox{ with } b_1,b_2\in\mathbb{C}^{n+2}.
\]
 Then $Y$ is immersed at the points $(b_1^t+b_2^t)(\bar{b}_1+\bar{b}_2)>0$ and $\hat{Y}$ is immersed  at the points  $(b_1^t-b_2^t)(\bar{b}_1-\bar{b}_2)>0$.
\end{theorem}

 \section{Isotropic harmonic maps into $SO^+(1,n+3)/(SO^+(1,1)\times SO(n+2))$}
\label{section3}

In this section we will recall briefly the DPW construction of harmonic maps and applications to the isotropic harmonic maps related to Willmore surfaces. We refer to \cite{Helein, Helein2, Xia-Shen, BW} for more details.

\subsection{The DPW construction of harmonic maps}

\subsubsection{Harmonic maps into an inner symmetric space}

Let $G/K$ be an inner symmetric space with involution $\sigma: G\rightarrow G$ such that $G^{\sigma}\supset K\supset(G^{\sigma})_0$. Let $\pi:G\rightarrow G/K$ be the projection of $G$ into $G/K$. Let $\mathfrak{g}=Lie(G)$ and $\mathfrak{k}=Lie(K)$ be their Lie algebras. We have the Cartan decomposition
$\mathfrak{g}=\mathfrak{k}\oplus\mathfrak{p}, $
$[\mathfrak{k},\mathfrak{k}]\subset\mathfrak{k},$
$[\mathfrak{k},
\mathfrak{p}]\subset\mathfrak{p},$
$[\mathfrak{p},\mathfrak{p}]\subset\mathfrak{k}.
$

Let $ \mathcal{F}:M\rightarrow G/K$ be a conformal harmonic map from a Riemann surface $M$, with $U\subset M$  an open connected subset.
Then there exists a frame $F: U\rightarrow G$ such that $ \mathcal{F} =\pi\circ F$. So we have the Maurer-Cartan form $F^{-1}\dd F= \alpha,$ and  Maurer-Cartan equation
$  \dd \alpha+\frac{1}{2}[\alpha\wedge\alpha]=0.$
Set
$
\alpha=\alpha_0+\alpha_1,$ with  $\alpha_0\in \Gamma(\mathfrak{k}\otimes T^*M), ~
\alpha_1\in \Gamma(\mathfrak{p}\otimes T^*M).$
 Decompose $\alpha_1$ further into the $(1,0)-$part $\alpha_1^\prime$ and the $(0,1)-$part $\alpha_1^{\prime \prime}$. Then set
$
\alpha_{\lambda}=\lambda^{-1}\alpha_{1}^\prime+\alpha_0+\lambda\alpha_{1}^{\prime \prime}, $ with $ \lambda\in  S^1.
$
We have the well-known characterization of harmonic maps:
\begin{lemma} $($\cite{DPW}$)$ The map  $ \mathcal{F}:M\rightarrow G/K$ is harmonic if and only if
\[
 \dd \alpha_{\lambda}+\frac{1}{2}[\alpha_{\lambda}\wedge\alpha_{\lambda}]=0\ \ \hbox{for all}\ \lambda \in  S^1.
\]
\end{lemma}
As a consequence, for a harmonic map $f$, the equation $
\dd F(z,\lambda)= F(z, \lambda) \, \alpha_{\lambda}
$
with  $F(0,\lambda)=F(0)$, always has a solution, which is   called the {\em extended frame}
 of  $ \mathcal{F}$.

\subsubsection{Two decomposition theorems}

We denote by $SO^+(1,n+3)$ the connected component of the identity of the linear isometry group of $\mathbb{R}^{n+4}_1$.
 Then
\[
\mathfrak{s}o(1,n+3)=\mathfrak{g}=\{X\in \mathfrak{g}l(n+4,\mathbb{R})|X^tI_{1,n+3}+I_{1,n+3}X=0\}.
\]
Define the involution
 \[
\begin{array}{ll}
\sigma:  SO^+(1,n+3)& \rightarrow SO^+(1,n+3)\\
 \ \ \ \ \ \ \ A&\mapsto DAD^{-1},
\end{array}
\quad \quad \hbox{where} \quad
D=\left(
         \begin{array}{ccccc}
             -I_{2} & 0 \\
            0 & I_{n+2} \\
         \end{array}
       \right).
\]
We have $SO^+(1,n+3)^{\sigma}\supset SO^+(1,1)\times SO(n+2)= (SO^+(1,n+3)^{\sigma})_0$. We also have
\[
\mathfrak{g}=
\left\{\left(
                   \begin{array}{cc}
                     A_1 & B_1 \\
                     -B_1^tI_{1,1} & A_2 \\
                   \end{array}
                 \right)
 |A_1^tI_{1,1}+I_{1,1}A_1=0,A_2+A_2^t=0\right\}=\mathfrak{k}\oplus\mathfrak{p},
\]
with\[
\mathfrak{k}=\left\{\left(
                   \begin{array}{cc}
                     A_1 &0 \\
                     0 & A_2 \\
                   \end{array}
                 \right)
 |A_1^tI_{1,1}+I_{1,1}A_1 =0, A_2+A_2^t=0\right\}, ~~
\  \mathfrak{p}=\left\{\left(
                   \begin{array}{cc}
                   0 & B_1 \\
                     -B_1^tI_{1,1} & 0 \\
                   \end{array}
                 \right)
\right\}.
\]
Let \[G^{\mathbb{C}}=SO^+(1,n+3,\mathbb{C}) := \{X \in SL(n+4,\C) ~|~ X^t I_{1,n+3} X =I_{1,n+3}\}, ~ \mathfrak{g}^{\mathbb{C}}=\mathfrak{so}(1,n+3,\mathbb{C}).\]
 Extend $\sigma$ to an inner involution of $G^{\mathbb{C}}$  with fixed point group $K^{\mathbb{C}}=S(O^+(1,1,\mathbb{C})\times O(n+2,\C))$.

Let $\Lambda G^{\mathbb{C}}_{\sigma}$ denote the group of loops in $G^C =SO^+(1,n+3,\mathbb{C})$  twisted by $\sigma$. Let $\Lambda^+G^{\mathbb{C}}_{\sigma}$  denote the subgroup of loops which extend holomorphically to the unit disk $|\lambda|\leq1$. We also
need the subgroup
$
\Lambda_B^+ G^{\mathbb{C}}_{\sigma}:=\{\gamma\in\Lambda^+G^{\mathbb{C}}_{\sigma}~|~\gamma|_{\lambda=0}\in \mathfrak{B} \}
$, where
$\mathfrak{B}\subset K^{\mathbb{C}}$ is defined from the Iwasawa decomposition
$K^{\mathbb{C}}=K\cdot\mathfrak{B}.
$
In this case,
\[
\mathfrak{B}=\left\{\left(
                   \begin{array}{cc}
                     \mathrm{b}_1 & 0 \\
                     0 & \mathrm{b}_2 \\
                   \end{array}
                 \right)\ \left|\ \mathrm{b}_1=\left(
                        \begin{array}{cc}
                          \cos\theta & i\sin\theta \\
                          i\sin\theta & \cos\theta \\
                        \end{array}
                      \right), \theta\in  \real \mod{2\pi{\mathbb  Z}},\ ~\hbox{and } ~\mathrm{b}_2\in \mathfrak{B}_2\right.\right
\}.
\]
Here  $\mathfrak{B}_2$  is the solvable subgroup of $SO(n+2,\mathbb{C})$ such that $SO(n+2,\mathbb{C})=SO(n+2)\cdot\mathfrak{B}_2$. We refer to
 Lemma 4 of \cite{Helein} for more details.
\begin{theorem}\label{thm-iwasawa} Theorem 5 of \cite{Helein}, see also \cite{Xia-Shen},  \cite{DPW}, \cite{PS}  (Iwasawa decomposition):
The multiplication $\Lambda G_{\sigma}\times \Lambda^+_{B}G^{\mathbb{C}}\rightarrow\Lambda G^{\mathbb{C}}_{\sigma}$ is a real analytic diffeomorphism onto the open dense subset $\Lambda G_{\sigma}\cdot \Lambda^+_{B}G^{\mathbb{C}} \subset\Lambda G^{\mathbb{C}}_{\sigma}$.
\end{theorem}

Let $\Lambda^-_*G^{\mathbb{C}}_{\sigma}$ denote the loops that extend holomorphically into $\infty$ and take values $I$ at infinity.
\begin{theorem} \label{thm-birkhoff}Theorem 7 of \cite{Helein}, see also \cite{Xia-Shen}, \cite{DPW}, \cite{PS}  (Birkhoff decomposition):
The multiplication $\Lambda^-_* G^{\mathbb{C}}_{\sigma}\times \Lambda^+G^{\mathbb{C}}\rightarrow\Lambda G^{\mathbb{C}}_{\sigma}$ is a real analytic diffeomorphism onto the open subset $ \Lambda^-_* G^{\mathbb{C}}_{\sigma}\cdot \Lambda^+G^{\mathbb{C}}$ (the big cell) of $\Lambda G^{\mathbb{C}}_{\sigma}$.

\end{theorem}

\subsubsection{The DPW construction and Wu's formula}

Here we recall the DPW construction for harmonic maps. Let $\mathbb{D}\subset\mathbb{C}$ be a disk or $\mathbb{C}$ itself, with complex coordinate $z$.
\begin{theorem} \label{thm-DPW} \cite{DPW}
\begin{enumerate}
\item Let $ \mathcal{F}:\mathbb{D}\rightarrow G/K$ be a harmonic map with an extended frame $F(z,\bar{z},\lambda)\in \Lambda G_{\sigma}$ and $F(0,0,\lambda)=I$. Then there exists a Birkhoff decomposition
\[
F_-(z,\lambda)=F(z,\bar{z},\lambda)F_+(z,\bar{z},\lambda),~ \hbox{ with }~ F_+\in\Lambda^+G^{\mathbb{C}}_{\sigma},
\]
such that $F_-(z,\lambda):\mathbb{D} \rightarrow\Lambda^-_*G^{\mathbb{C}}_{\sigma}$ is meromorphic. Moreover, the Maurer-Cartan form of $F_-$ is of the form
\[
\eta=F_-^{-1}\dd F_-=\lambda^{-1}\eta_{-1}(z)\dd z,
\]
with $\eta_{-1}$ independent of $\lambda$. The $1$-form $\eta$ is called the normalized potential of $ \mathcal{F}$.
\item Let $\eta$ be a $\lambda^{-1}\cdot\mathfrak{p}-$valued meromorphic 1-form on $\mathbb{D}$. Let $F_-(z,\lambda)$ be a solution to $F_-^{-1}\dd F_-=\eta$, $F_-(0,\lambda)=I$. Then on an open subset $\mathbb{D}_{\mathfrak{I}}$ of $\mathbb{D}$ one has
\[
F_-(0,\lambda)=\tilde{F}(z,\bar{z},\lambda)\cdot \tilde{F}_+(z,\bar{z},\lambda),\ \hbox{ with }\ \tilde{F}\in\Lambda G_{\sigma},\ \tilde{F}_+\in\Lambda ^+_{B} G^{\mathbb{C}}_{\sigma}.
\]
This way, one obtains an extended frame $\tilde{F}(z,\bar{z},\lambda)$ of some harmonic map from  $\mathbb{D}_{\mathfrak{I}}$  to $G/K$ with $\tilde{F}(0,\lambda)=I$. Moreover, all harmonic maps can be obtained in this way, since these two procedures are inverse to each other if the normalization at some based point is used.
\end{enumerate}
\end{theorem}
The normalized potential can be determined in the following way. Let  $f$ and $F$ be as above. Let $\alpha_{\lambda}=F^{-1}\dd F$. Let $\delta_1$ and $\delta_0$ denote the sum of the holomorphic terms of $z$ about $z=0$ in the Taylor expansion of $\alpha_1^\prime (\frac{\partial}{\partial z})$ and  $\alpha_0^\prime (\frac{\partial}{\partial z})$.
 \begin{theorem} \label{thm-wu} \cite{Wu} (Wu's formula) We retain the notations in Theorem \ref{thm-DPW}. The the normalized potential of $ \mathcal{F}$ with respect to the base point $0$ is given by
\begin{equation}\label{eq-Wu}
 \eta=\lambda^{-1}F_0(z)\delta_1F_0(z)^{-1} \dd z,
\end{equation}
where $F_0(z):\mathbb{D}\rightarrow G^{\mathbb{C}}$ is the solution to $F_0(z)^{-1}\dd F_0(z)=\delta_0 \dd z$, $F_0(0)=I$.
\end{theorem}
 \begin{lemma}\label{lemma-conj}   We retain the notations in Theorem \ref{thm-wu}. Let $Q\in K$ and $Q\mathcal{F}$ be a transform of $\mathcal{F}$ in $G/K$. Then the normalized potential of $Q\mathcal{F}$ with respect to the base point $0$ is
 \[\eta_{Q}=Q\eta Q^{-1}.\]
\end{lemma}
\begin{proof}
We have now the lift $QFQ^{-1}$  of $Q\mathcal{F}$ with respect to the base point $0$. So we have the Birkhoff splitting of  $QFQ^{-1}$ as below
\[F_{-Q}=QF_-Q^{-1}=QFQ^{-1}QF_+Q^{-1} \hbox{ since } F_-=FF_+.\]
Hence we obtain
\[\eta_{Q}=(QF_-Q^{-1})^{-1}\dd QF_-Q^{-1}=Q\eta Q^{-1}.\]
\end{proof}
This lemma shows that we can identify the normalized potentials up to an conjugation of elements in $K$.

\subsection{Potentials of isotropic harmonic maps}

\subsubsection{The general case}
  Let  $\mathbb{D}$ denote the unit disk of $\mathbb{C}$ or $\mathbb{C}$ itself.
Let $ \mathcal{F}: \mathbb D \rightarrow SO^+(1,n+3)/(SO^+(1,1)\times SO(n+2))$
 be a  harmonic map with a   lift  $F:\mathbb D\rightarrow SO^+(1,n+3)$ and the Maurer-Cartan form $\alpha=F^{-1}\dd F$. Then
\[
\alpha_0^\prime=\left(
                   \begin{array}{cc}
                     A_1 &0 \\
                     0 & A_2 \\
                   \end{array}
                 \right) \dd z,\ \ \alpha_1^\prime=\left(
                   \begin{array}{cc}
                     0 & B_1 \\
                     -B_1^tI_{1,1} & 0 \\
                   \end{array}\right) \dd z.
\]

  \begin{theorem}(\cite{Helein,Helein2,Xia-Shen,BW}) The normalized potential of an isotropic harmonic map $ \mathcal{F}=Y\wedge\hat{Y}$ is of the form

 \begin{equation}\label{eq-b1 of a}
\eta=\lambda^{-1}\left(
                   \begin{array}{cc}
                    0 & \hat{B}_1 \\
                     -\hat{B}_1^tI_{1,1} & 0\\
                   \end{array}
                 \right)\dd z,
\hbox{   with }\ \hat{B}_1\hat{B}_1^t=0.
\end{equation}
Moreover, $[Y]$ and $[\hat{Y}]$ forms a pair of dual (S-)Willmore surfaces if and only if $rank(
\hat{B}_1)=1$.

Conversely, let $ \mathcal{F}=Y\wedge \hat{Y}$ be an harmonic map with  normalized potential
\[
\eta=\lambda^{-1}\left(
                   \begin{array}{cc}
                    0 & \hat{B}_1 \\
                     -\hat{B}_1^tI_{1,1} & 0\\
                   \end{array}
                 \right)\dd z
\]
                 satisfying \eqref{eq-b1 of a}. Then $ \mathcal{F}$ is an isotropic harmonic map.
\end{theorem}

\subsubsection{On minimal surfaces in space forms}
In \cite{Helein}, there is an interesting description of Willmore surfaces
M\"obius equivalent to minimal surfaces in space forms. Here we restate it as:
\begin{theorem}\label{thm-minimal}(\cite{Helein}, \cite{Xia-Shen}) Let $\mathcal{F}_h=Y\wedge \hat{Y}$ be a non-constant isotropic harmonic map.
\begin{enumerate}
\item The map $[Y]$
 is M\"obius equivalent to  a minimal surface in $\mathbb{R}^{n+2}$ if $\hat{Y}$ reduces to a point. In this case
$
B_1= \left(
          \begin{array}{cc}
            b_1 &
            b_1  \\
          \end{array}
        \right)^t.
$
\item
The map $[Y]$ is M\"obius equivalent to  a minimal surface in $ S^{n+2}$ if $\mathcal{F}_h$ reduces to a harmonic map into $SO(n+3)/SO(n+2)$. In this case
$
B_1= \left(
          \begin{array}{cc}
            0 &
            b_1 \\
          \end{array}
        \right)^t.
$
\item
The map $[Y]$ is M\"obius equivalent to  a minimal surface in $H^{n+2}$  if $\mathcal{F}_h$ reduces to a harmonic map into $SO^+(1,n+2)/SO^+(1,n+1)$. In this case
$
B_1= \left(
          \begin{array}{cc}
            b_1 &
            0\\
          \end{array}
        \right)^t .
$
\end{enumerate}
Here $b_1$ takes values in $\mathbb{C}^{n+2}$ and satisfies $b_1^tb_1=0$.

The converse of the above results also hold. That is, if $B_1$ is (up to conjugation) of the form stated above, then $[Y]$ is M\"obius equivalent to  the corresponding minimal surface where it is an immersion.
\end{theorem}

\subsection{On harmonic maps of finite uniton type}

In this subsection we will discuss harmonic maps of finite uniton type.

Loops which have a finite Fourier expansion will be called {\it algebraic loops} and
the corresponding spaces will be denoted by the subscript $``alg"$, like $\Lambda_{alg} G_{\sigma},\ \Lambda_{alg} G^{\mathbb{C}}_{\sigma},\
\Omega_{alg} G_{\sigma}.$
 We define
$$ \Omega^k_{alg} G_{\sigma}:=\{\gamma\in
\Omega_{alg} G_{\sigma}|
Ad(\gamma)=\sum_{|j|\leq k}\lambda^jT_j \}\ .$$
Let $G/K$ be an inner symmetric space (given by the inner involution $\sigma:G\rightarrow G$). We map $G/K$ into $G$ as totally geodesic submanifold via the (finite covering) Cartan map:
$ \mathfrak{C} :G/K  \rightarrow G,
  \mathfrak{C} (gK) =g\sigma(g)^{-1}.$

\begin{definition} {\em(\cite{Uh,BuGu, DoWa1, DoWa2})}
\begin{enumerate}
\item Let $f:M \rightarrow G$  be a harmonic map into a real Lie group $G$ with
extended solution $\Phi(z,\lambda)\in \Lambda G^{\mathbb{C}}_{\sigma}$.
We say that $f$ has {\it finite uniton number k} if
$$\Phi(M)\subset \Omega^k_{alg} G_{\sigma},\
\hbox{ and } \Phi(M)\nsubseteq \Omega^{k-1}_{alg} G_{\sigma}.$$

\item
 A harmonic map $f$ into $G/K$  is said to be of finite uniton number $k,$ if it is of finite uniton number $k,$ when considered as a harmonic map into $G$ via the Cartan map, i.e., $f$ has finite uniton number $k$ if and only if $\mathfrak{C} \circ f$ has finite uniton number $k$.
 \end{enumerate}
\end{definition}

It is proved that for harmonic maps into inner symmetric space $G/K$, it is of finite uniton number if and only if its normalized potential takes value in some nilpotent Lie sub-algebra \cite{BuGu,Gu2002,DoWa2}. In Section 4 and Section 5 we will give a characterization of totally isotropic Willmore two-spheres in terms of harmonic maps of finite uniton number at most $2$.

\section{Totally isotropic Willmore two-spheres and their adjoint transforms}

In this section we will first collect the geometric results concerning totally isotropic Willmore two-spheres and their adjoint transforms. Then by the geometric descriptions, we are able to derive the normalized potentials of the isotropic harmonic map given by such Willmore surfaces and their adjoint transforms.
\subsection{Totally isotropic Willmore surfaces}
Let $y:M\rightarrow S^{2m}$  be a conformal immersion and we retain the notion in Section 2.
Then \cite{Calabi,Ejiri1988}   $y$ is called {\em totally isotropic} if and only if all the derivatives of $y$ with respect to $z$ are isotropic, or equivalently,
\[ \langle Y_z^{(j)}, Y_{z}^{(l)} \rangle =0 ~~\hbox{ for all } j,\ l \in \mathbb{Z}^+.\] Here ``$Y_{z}^{(j)}$" denotes taking $j$ times derivatives of $Y$ by $z$. As a consequence a totally isotropic Willmore surface always locates in an even dimensional sphere.
  Moreover, we can find locally an isotropic frame $\{E_j\}$, ${j=1,\cdots,m}$, such that
\begin{equation}\label{eq-HTI-0}\begin{split}
  &\langle E_j, \bar{E}_l\rangle=2\delta_{jl},\ j,l=1,\cdots,m,\\
 & Y_z \in Span_{\C}\{E_1 \} \mod Y,\\
  & Y_z^{(j)}\in Span_{\C}\{E_j\}\mod Span_{\C}\{Y, E_1,\cdots,E_{j-1}\},\ j=2,\cdots,m,\\
&\hbox{$\{E_j,\bar {E}_j\}_{j=2,\cdots,m}$ forms a basis of the normal bundle $V^{\perp}$.}
\end{split}\end{equation}
Next we call $y$ an {\em H-totally isotropic surface} if it satisfies furthermore the following conditions
\begin{equation}\label{eq-HTI}
  D_{\bar z} E_j\in Span_{\C}\{E_2,\cdots,E_m\}, \hbox{ for all } j=2,\cdots,m.
\end{equation}
It is direct to verify that this condition is independent of the choice of $z$, $Y$ and $E_j$. Here  the notion ``H" comes from two facts. First,  this condition is similar to the horizontal conditions for minimal two-spheres in $S^{2m}$ \cite{Bryant1982,Calabi}. Second, by a result of \cite{Ma}, we can prove that
\begin{theorem}\cite{MaWang}
Let $y$ be a totaly isotropic Willmore two-sphere  in $S^{2m}$. Then $y$ is an H-totally isotropic surface.
\end{theorem}See \cite{Wang-iso} for a proof in the case $m=3$. The key point of the proof is an application of the holomorphic forms given by $\kappa$ and its derivatives, which can be found in Section 5 of \cite{Ma}.

Moreover, totally isotropic surfaces in $S^{2m}$ may not be Willmore. But H-totally isotropic surfaces must be Willmore.
\begin{proposition}
Let $y$ be an H-totally isotropic surface in $S^{2m}$. Then $y$ is Willmore.
\end{proposition}
\begin{proof}
By definition of $E_j$, we have $\kappa\in Span_{\C}\{E_2\}$.
From \eqref{eq-HTI}, we have that
\[D_{\bar z}\kappa\in Span_{\C}\{E_2,\cdots,E_m\},\ D_{\bar z}D_{\bar z}\kappa\in Span_{\C}\{E_2,\cdots,E_m\}.\]
So $D_{\bar z}D_{\bar z}\kappa+\frac{\bar s}{2}\kappa \in Span_{\C}\{E_2,\cdots,E_m\}$. Hence $Im(D_{\bar z}D_{\bar z}\kappa+\frac{\bar s}{2}\kappa)=0$ in \eqref{eq-integ} indicates that the Willmore equation $D_{\bar z}D_{\bar z}\kappa+\frac{\bar s}{2}\kappa=0$ holds, i.e., $y$ is Willmore.
\end{proof}

Concerning the adjoint surfaces of  H-totally isotropic surfaces, we have
\begin{proposition} \cite{MaWang}
Let $y$ be an H-totally isotropic surface in $S^{2m}$ (hence Willmore). Then the adjoint surface of  $y$ is also H-totally isotropic surface on the points is is immersed.
\end{proposition}
This can be easily derived since $\hat Y_{z}\in Span_{\C}\{E_1,\cdots,E_m\}\mod \{Y, \hat Y\}$ by \eqref{eq-hat-Y-z}.

\subsection{Normalized potentials of H-totally isotropic surfaces}

The normalized potentials of H-totally isotropic surfaces can be derived from Wu's formula as below
\begin{theorem}\label{thm-main-1}
Let $y:\mathbb D\rightarrow S^{2m}$ be an  H-totally isotropic surface   with a local adjoint transform $\hat y=[\hat Y]$. Assume that $\mathcal{F}|_{z=0}=I\mod K$. Then up to an conjugation, the normalized potential of $ \mathcal{F} =Y\wedge \hat Y$ has the form
\begin{equation}\label{eq-H-iso-NP}
\eta=\lambda^{-1}\left(
                   \begin{array}{cc}
                    0 & \hat{B}_1 \\
                     -\hat{B}_1^tI_{1,1} & 0\\
                   \end{array}
                 \right)\dd z, \ ~~\hat {B}_1=\left(
      \begin{array}{ccccccc}
                          h_{11} & i h_{11}  & \cdots &   h_{m1} & i h_{m1} \\
                         \hat{h}_{11} & i\hat{ h}_{11}  & \cdots &   \hat{h}_{m1} & i \hat{h}_{m1} \\
      \end{array}
    \right),
\end{equation}
with $\{ h_{j1}\dd z,\  \hat h_{j1}\dd z|\ j=1,\cdots,m \}$  being meromorphic 1-forms on $\mathbb D$.
 \end{theorem}

\begin{proof}
We have the following
\[Y_z=-\frac{\mu}{2}Y+\frac{1}{2}E_{1}.\]
We consider the lift $\tilde{F}$ as below
\[\tilde{F}=\left(\frac{1}{\sqrt{2}}(Y+\hat{Y}),\frac{1}{\sqrt{2}}(-Y+\hat{Y}),e_1,\psi_2,\cdots,\psi_m,\hat e_1,\hat\psi_2,\cdots,\hat\psi_m\right).\]
Here we use the frame defined in \eqref{eq-HTI-0} and set \[E_1=e_1+i\hat e_1,\ E_j=\psi_j+i\hat\psi_j,\ j=2,\cdots, m .\] Set
\[\kappa=\sum_jk_j(\psi_j+i\hat\psi_j),\ \zeta=\sum_{j }\gamma_j(\psi_j+i\hat\psi_j).\]
Assume that $D_{z}E_j=\sum a_{jl}E_l$, $D_{z}\bar{E}_j=\sum \hat{a}_{jl}\bar{E}_l$. By \eqref{eq-HTI-0}, we have $ a_{jl}+\hat{a}_{lj}=0$, and
\[D_{z}\psi_j= \frac{1}{2} \left(\sum(a_{jl}-{a}_{lj})\psi_l+\sum i(a_{jl}+{a}_{jl})\hat \psi_l\right),\]
\[D_{z}\hat\psi_j =\frac{1}{2} \left(-\sum i(a_{jl}+{a}_{lj})\psi_l+\sum (a_{jl}-{a}_{lj})\hat \psi_l\right).\]
So
\[A_2= \left(
         \begin{array}{cc}
           A_{21} & iA_{22} \\
           -iA_{22}^t & A_{21} \\
         \end{array}
       \right) \ \hbox{ with }  \ A_{12}+A_{12}^t=0,\ A_{22}=A_{22}^t,\]
and       \[ A_{21}=\left(
                   \begin{array}{ccc}
                     0 & -\frac{k_l}{2}  \\
                     \frac{k_j}{2} &  \frac{a_{jl}-{a}_{lj}}{2} \\
                   \end{array}
                 \right)_{m\times m},
        \ A_{22}=\left(
                   \begin{array}{ccc}
                     -\frac{\mu}{2} & -\frac{k_l}{2}  \\
                     \frac{k_j}{2} & \frac{a_{jl}+{a}_{lj}}{2}  \\
                   \end{array}
                 \right)_{m\times m}.
\]
Hence, without lose of generality, we assume that $\tilde{F}(0,0,\lambda)=I$ and let
\[\delta_0= \left(
         \begin{array}{cc}
           \check{A}_{1} & 0\\
           0 & \check{A}_{2} \\
         \end{array}
       \right),\ \delta_1= \left(
         \begin{array}{cc}
           0 & \check{B}_1\\
           -\check{B}_1^tI_{1,1} & 0 \\
         \end{array}
       \right)\]
       be the holomorphic parts of $\tilde\alpha_0^\prime (\frac{\partial}{\partial z})$ and $\tilde\alpha_1^\prime (\frac{\partial}{\partial z})$ respectively. Then we have
       \begin{equation}\label{eq-U(n)}
          \check{A}_{2}=\left(
         \begin{array}{cc}
           \check{A}_{21} & i\check{A}_{22} \\
           -i\check{A}_{22}^t & \check{A}_{21} \\
         \end{array}
       \right)\ \hbox{ with } \check{A}_{21}^t+\check{A}_{21}=0,\ \check{A}_{22}^t=\check{A}_{22},
       \end{equation}
       and
       \[\check{B}_1=\left(
                       \begin{array}{cc}
                         \mathbf{b}_1^t & i\mathbf{b}_1^t \\
                         \hat{\mathbf{b}}_1^t & i\hat{\mathbf{b}}_1^t \\
                       \end{array}
                     \right).
       \]
Let $F_0(z):\mathbb{D}\rightarrow G^{\mathbb{C}}$ be the solution to $F_0(z)^{-1}\dd F_0(z)=\delta_0 \dd z$, $F_0(0)=I$. We see that
\[F_0=\left(
                       \begin{array}{cc}
                         F_{01} & 0 \\
                         0 & F_{02}\\
                       \end{array}
                     \right)\]
                     with $F_{01}=\exp(z \check{A}_{1})$ and
                     \[F_{02}=\exp(z \check{A}_{2})=\left(
         \begin{array}{cc}
           F_{021} & F_{022}  \\
           -F_{022}  & F_{021}  \\
         \end{array}
       \right),\ F_{02}^{-1}=F_{02}^t=\left(
         \begin{array}{cc}
           F_{021}^t & -F_{022}^t  \\
           F_{022}^t  & F_{021}^t  \\
         \end{array}
       \right),\]
       since $\check{A}_{2}$ satisfies \eqref{eq-U(n)}. So the normalized potential has the form by Wu's formula \eqref{eq-Wu}
       \[ \widetilde{\eta}=\lambda^{-1}\left(
                   \begin{array}{cc}
                    0 &  \widetilde{B}_1 \\
                     -\widetilde{B}_1^tI_{1,1} & 0\\
                   \end{array}
                 \right)\dd z~~ \hbox{ with }~~\widetilde{B}_1 =F_{01}\check{B}_1F_{02}^{-1}.\]
             So
             \[ \widetilde{B}_1 =F_{01}\check{B}_1F_{02}^{-1}=F_{01}\left(
                       \begin{array}{cc}
                         \mathbf{b}_1^tF_{021}^t+ i\mathbf{b}_1^tF_{022}^t& i(\mathbf{b}_1^tF_{021}^t+ i\mathbf{b}_1^tF_{022}^t) \\
                         \hat{\mathbf{b}}_1^tF_{021}^t+ i\hat{\mathbf{b}}_1^tF_{022}^t& i(\hat{\mathbf{b}}_1^tF_{021}^t+ i\hat{\mathbf{b}}_1^tF_{022}^t) \\
                       \end{array}
                     \right)=\left(
                       \begin{array}{cc}
                          b _1^t & i b_1^t \\
                         \hat{ {b}}_1^t & i\hat{ {b}}_1^t \\
                       \end{array}
                     \right).\]
By a conjugation of (see Lemma \ref{lemma-conj})
\[Q=\left(
      \begin{array}{cc}
      I_2 & 0 \\
        0 & Q_2 \\
      \end{array}
    \right)\ \ \hbox{ with }\
Q_2=\left(
      \begin{array}{cccccccc}
        1 & 0 &   &  &   &  &   &  \\
        &  &1 & 0 &   &  &   &  \\
       &  &   &  & \cdots &  &   &  \\
         &  &   &  &   &  &   1& 0 \\
        0 & 1 &   &  &   &  &   &  \\
          &   & 0 & 1 &   &  &   &  \\
       &  &   &  & \cdots &  &   &  \\
         &  &   &  &   &  &   0& 1 \\
      \end{array}
    \right),
\]
we see that the normalized potential $\eta=Q^{-1}\widetilde{\eta} Q$ has the desired form \eqref{eq-H-iso-NP}.

\end{proof}

The converse part of Theorem \ref{thm-main-1} needs a detailed discussion of the Iwasawa decompositions of $F_-$, see the next Section.

\section{Potentials corresponding to  H-totally isotropic surfaces}

In this section we will first give a characterization of H-totally isotropic surfaces in terms of normalized potentials. This also provides a procedure to construct examples. As illustrations we derive two kinds of examples. We will state the main results in this section and leave the computations to the next section.

\subsection{The characterization of H-totally isotropic surfaces}
\begin{theorem}\label{thm-main-2} Let $\eta$ be a normalized potential of the form \eqref{eq-H-iso-NP}. Let $ \mathcal{F} =Y\wedge \hat Y$ be the corresponding isotropic harmonic map. Then  $[Y]$ and $[\hat Y]$ are a pair of $H-$totally isotropic adjoint Willmore surfaces on the points they are immersed. And  $\mathcal{F}$ is a harmonic map of finite uniton number at most $2$.
\end{theorem}

To prove Theorem \ref{thm-main-2}, one need to perform an Iwasawa decomposition. A simple way to do this is to make the potential in \eqref{eq-H-iso-NP} being of strictly upper-triangle matrices. For this purpose, we will need a Lie group isometry. Then under this isometry, we can write down the Iwasawa decomposition in an explicit way. As a consequence, we can derive some geometric properties of the corresponding Willmore surfaces.

First we define a new Lie group as below
\begin{equation}
G(n,\mathbb{C})=\{A\in Mat(n,\mathbb{C})| A^tJ_nA=J_n, \det A=1
\},~ \hbox{with }~~J_n=\left(
      \begin{array}{ccc}
          &   & 1 \\
         & \iddots &  \\
        1 &  &   \\
      \end{array}
    \right).
\end{equation}

Theorem \ref{thm-main-2} can be derived from the following lemmas.

\begin{lemma}\label{lemma-gm} We have the Lie group isometry
 \begin{equation}\label{eq-Lie}
\begin{array}{ll}
 \mathcal{P}:     SO(1,2m+1,\mathbb{C})&\rightarrow  G(2m+2,\mathbb{C})\\
 \ \ \ \ \ \ \ A&\mapsto\tilde{P}_1^{-1} \tilde{P}^{-1}A\tilde{P}\tilde{P}_1,
\end{array}\end{equation}
with \[\tilde{P}= \frac{1}{\sqrt{2}}\left(
    \begin{array}{cccccccc}
      1 &  &   &  &   &   &   & -1  \\
       1 &  &    &   &   &   & & 1   \\
        & -i &   &  &  &   &  i &   \\
       &  1&   &  &   &   &   1&   \\
       &  &  \ddots  &  &     &  \iddots   &   &   \\
       &  &    & -i & i  &    &   &   \\
       &  &   & 1 &  1 &    &   &   \\
    \end{array}
  \right) \ \hbox{ and } \tilde{P}_1=\left(
                                                \begin{array}{cccc}
                                                  0 & 1 & 0 & 0\\
                                                  I_{m} & 0 & 0 & 0 \\
                                                  0 & 0 & 0 & I_{m} \\
                                                  0 & 0 & 1 & 0 \\
                                                \end{array}
                                              \right)
  .\]
Moreover, we have the following results.
\begin{enumerate}
\item $\mathcal{P}\left(SO(1,2m+1)\right)=\left\{B\in G(2m+2,\mathbb{C})| B=S_0\bar{B}S_0^{-1}\right\}$
with
\[S_0= \tilde{P}_1^{-1}\tilde{P}^{-1} \bar{\tilde{P}}\bar{\tilde{P}}_1  
                                                 = \left(
                                                \begin{array}{cccc}
                                                  0 & 0 &  J_{m}\\
                                                  0 & I_2 & 0 \\
                                                  J_{m} & 0 & 0  \\
                                                \end{array}
                                              \right).\]
So
$\mathcal{P}\left(\Lambda SO(1,2m+1)\right)=\{F\in\Lambda G(2m+2,\mathbb{C})|
\tau(F)=F\}$ with
 \begin{equation}\begin{array}{ll}
\tau:    G(2m+2,\mathbb{C}) &\rightarrow G(2m+2,\mathbb{C})\\
 \ \ \ \ \ \ \ F &\mapsto S_0\bar{F}S_0^{-1}.
\end{array}\end{equation}
And $\tau(F)^{-1}=\hat{J}\bar{F}^t\hat{J}^{-1}$ with
\[\hat{J}=\left(
            \begin{array}{ccc}
              I_{m} &   &  \\
               & J_{2} &  \\
               &  &  I_{m} \\
            \end{array}
          \right).\]
\item  The image of the subgroup $K^{\mathbb{C}}=SO(1,1,\C)\times SO(2m,\C)$ in $G(2m+2,\mathbb{C})$ is
\[\mathcal{P}(K^{\mathbb{C}}) =\{B\in G(2m+2,\mathbb{C})|\ B=D_0BD_0^{-1}\},\]
with
$ D_0= \tilde{P}_1^{-1}\tilde{P}^{-1}D\tilde{P} \tilde{P}_1=diag\left( I_{m},  -I_{2},  I_{m}
                               \right)$, $D=diag\left( -I_2,  I_{2m} \right)$.
\end{enumerate}

\end{lemma}

\begin{lemma}\label{lemma-gm-2}  Under the isometry of \eqref{eq-Lie}, we have the following results
\begin{enumerate}
\item For $\eta_{-1}$ in  \eqref{eq-H-iso-NP}, one has
 \begin{equation}\label{eq-P-eta}
\mathcal{P}(\eta_{-1})=\left(
                             \begin{array}{ccc}
                               0 & \check{f} & 0 \\
                               0 & 0 & -J_m\check{f}^tJ_2 \\
                               0 & 0 & 0 \\
                             \end{array}
                           \right), \hbox{ with } \check{f}=\left(
      \begin{array}{ccccccc}
                     \check{f}_{11} & \check{f}_{12} \\
                     \vdots & \vdots \\
                     \check{f}_{m1} & \check{f}_{m2} \\
      \end{array}
    \right),~ \check{f}^{\sharp}:= J_{m}\check{f}^tJ_2, \end{equation}
    and
    \[\check{f}_{j1}=i(h_{j1}-\hat h_{j1}),\ \check{f}_{j2}=-i(h_{j1}+\hat h_{j1}), \ j=1,\cdots,m.\]
\item Let $H$ be a solution to $H^{-1}dH=\lambda^{-1}\mathcal{P}(\eta_{-1})\dd z$, $H(0,0,\lambda)=I$. Then
\begin{equation}\label{eq-H}H=I+\lambda^{-1}H_1+\lambda^{-2}H_2=\left(
        \begin{array}{ccc}
          I & \lambda^{-1}f & \lambda^{-2}g \\
          0 & I & -\lambda^{-1}f^{\sharp} \\
          0 & 0 & I \\
        \end{array}
      \right),
 \end{equation}
 with   $f^{\sharp}:= J_{m}f^tJ_2$, $f=\int_0^z\check{f} dz$ and $g=-\int_0^zf{\check{f}^{\sharp}} dz$.
 \item We have the Iwasawa decomposition of $H$ as follows
 \begin{equation}\label{eq-Iw}
   \tilde{F}=H\tau(W)L^{-1}=\left(
        \begin{array}{ccc}
         (I-f\bar{u}^{\sharp}-gJ\bar{v}J)l_1^{-1} & \lambda^{-1}(f+gJ\bar{u})l_0^{-1} & \lambda^{-2}gl_4^{-1}\\
          -\lambda (\bar{u}^{\sharp}J+f^{\sharp}J\bar{v}J)l_1^{-1}  & (I-f^{\sharp}J\bar{u})l_0^{-1} & -\lambda^{-1}f^{\sharp}l_4^{-1} \\
          \lambda^{2}J\bar{v}Jl_1^{-1}  & \lambda J \bar{u}l_0^{-1} & l_4^{-1} \\
        \end{array}
      \right).
 \end{equation}
 Here $W= I+\lambda^{-1}W_1+\lambda^{-2}W_2 $ and $ L=\hbox{diag}\{l_1,l_0,l_4\}$ satisfy
\[W_0=\left(
        \begin{array}{ccc}
          a & 0 &  0 \\
          0 & q & 0 \\
          0 & 0 &  \varrho \\
        \end{array}
      \right)=\tau(L)^{-1}L,\  W_1=\left(
        \begin{array}{ccc}
          0 & u & 0 \\
           0 & 0 & -u^{\sharp} \\
          0 & 0 & 0 \\
        \end{array}
      \right),\ W_2=\left(
        \begin{array}{ccc}
          0 & 0 & v \\
           0 & 0 & 0 \\
          0 & 0 & 0 \\
        \end{array}
      \right).\]
Here $a,q,d, u,v$ are solutions to the following equation
\begin{subequations} \label{eq-Iwasawa:1}
\begin{align}
&a+ uqJ\bar{u}^t+ v \varrho\bar{v}^t=I,\label{eq-Iwasawa:1A}\\
& uq-v\varrho \bar{u}^{\sharp t}J= f,\label{eq-Iwasawa:1B}\\
&v \varrho=g,\label{eq-Iwasawa:1C}\\
&q + u^{\sharp}\varrho\bar{u}^{\sharp t}J =I+J\bar{f}^tf,\label{eq-Iwasawa:1D}\\
&  u^{\sharp}\varrho =f^{\sharp}-J\bar{f}^tg,\label{eq-Iwasawa:1E}\\
& \varrho=I+\bar{f}^{t\sharp}J {f}^{ \sharp}+ \bar{g}^{t}g.\label{eq-Iwasawa:1F}
\end{align}
\end{subequations}
\item The M-C form of $\tilde{F}$ has the form
   \begin{equation}\label{eq-alpha-1}\tilde{\alpha}_1'=\left(
                             \begin{array}{ccc}
                               0 & l_1f'l_0^{-1} & 0 \\
                               0 & 0 & -l_0{f^{\sharp}}'l_4^{-1} \\
                               0 & 0 & 0 \\
                             \end{array}
                           \right)\dd z,\end{equation}
                           \begin{equation}\label{eq-alpha-2} \tilde{\alpha}_0'=\lambda^{-1}\left(
                             \begin{array}{ccc}
                               -f'  \bar{u}^{\sharp}J -l_{1z}l_1^{-1} & 0& 0 \\
                               0 & - \bar{u}^{\sharp}J f' -{f^{\sharp}}'J\bar{u}-l_{0z}l_0^{-1} & 0 \\
                               0 & 0 & -J\bar{u}{f^{\sharp}}'-l_{4z}l_4^{-1}\\
                             \end{array}
                           \right)\dd z.\end{equation}

\end{enumerate}

\end{lemma}

\begin{lemma}\label{lemma-gm-3}  Under the isometry of \eqref{eq-Lie}, $y$ is an $H-$totally isotropic surface   with an adjoint transform $\hat y=[\hat Y]$.
\end{lemma}
\begin{lemma}\label{lemma-gm-4}  Under the isometry of \eqref{eq-Lie}, $\mathcal{P}(\eta_{-1})$  takes value in the nilpotent Lie sub-algebra
\[\mathfrak{g}_{nil}=\left\{X\in \Lambda\mathfrak{g}(2m+2,\C) \left|X=\lambda^{-1}\left(
                                                                        \begin{array}{ccc}
                                                                          0 & \cdots & 0 \\
                                                                          0 & 0 & \cdots \\
                                                                          0 & 0 & 0 \\
                                                                        \end{array}
                                                                      \right)+\lambda^{-2}\left(
                                                                        \begin{array}{ccc}
                                                                          0 & 0& \cdots \\
                                                                          0 & 0 & 0 \\
                                                                          0 & 0 & 0 \\
                                                                        \end{array}
                                                                      \right)
\right.\right\}.\]
 As a consequence, $\mathcal{F}=Y\wedge\hat Y$ is  of finite uniton number at most $2$.
\end{lemma}

Here the proof of Lemma \ref{lemma-gm} is  a straightforward computation so that we leave it to the interested readers. Lemma \ref{lemma-gm-4} holds apparently.  The proof of Lemma \ref{lemma-gm-2} and \ref{lemma-gm-3} will be given in Section 6, since it involves many technical computations.

\subsection{Constructions of examples}
In this subsection we will provide two kinds of examples. The first one concerns the new Willmore two-sphere derived in \cite{DoWa1}, which is the first example of Willmore two-spheres in $S^6$ admitting no dual surfaces. Here we will derive this surface together with one of its adjoint surface. This also indicates that it can be derived from an adjoint transform of some minimal surface in $\mathbb R^6$.

The next example concerns one of the most simple Willmore two-spheres in $S^4$, i.e., the one derived by a holomorphic curve in $\C^2=\mathbb R^4$ with total curvature $-4\pi$.

\begin{theorem}\label{thm-example-1} Let
\[\eta=\lambda^{-1}\eta_{-1}\dd z\ \hbox{ with }   \hat B_1=\frac{1}{2}\left(
                                                           \begin{array}{cccccc}
                                                             -i & 1 & i &-1 &-2iz & 2z \\
                                                             i & -1 & i  &-1 & 2iz & -2z \\
                                                           \end{array}
                                                         \right).\]
Then the associated family of the corresponding isotropic harmonic maps is $Y\wedge \hat Y$, with
\begin{equation}\label{eq-Y-ex-1}
  Y=-\frac{\sqrt{2}}{2\varsigma} \left(
    \begin{array}{c}
      -1-r^2-\frac{r^4}{4}-\frac{r^6}{9}  \\
      1-r^2 +\frac{r^4}{4}+\frac{r^6}{9}  \\
      \frac{ir^2}{2}(\lambda^{-1}z-\lambda \bar{z})  \\
     - \frac{r^2}{2}(\lambda^{-1}z+\lambda \bar{z})  \\
      -i (\lambda^{-1}z-\lambda \bar{z}) \\
       \lambda^{-1}z+\lambda \bar{z} \\
       \frac{ir^2}{3}(\lambda^{-1}z^2-\lambda \bar{z}^2)  \\
      - \frac{r^2}{3}(\lambda^{-1}z^2+\lambda \bar{z}^2)   \\
    \end{array}
  \right),~~
~ \hat{Y}=\frac{\sqrt{2}}{2\varsigma}\left(
                          \begin{array}{c}
                             1+r^2+\frac{5r^4}{4}+\frac{4r^6}{9}+\frac{r^8}{36}  \\
                            1-r^2-\frac{3r^4}{4}+\frac{4r^6}{9}-\frac{r^8}{36}  \\
                            -i (\lambda^{-1}z-\lambda \bar{z})(1+\frac{r^6}{9})  \\
                             (\lambda^{-1}z+\lambda \bar{z})(1+\frac{r^6}{9})  \\
                            \frac{i r^2}{2}(\lambda^{-1}z-\lambda \bar{z})(1+\frac{4r^2}{3})  \\
                            -\frac{r^2}{2} (\lambda^{-1}z+\lambda \bar{z})(1+\frac{4r^2}{3})  \\
                            -i (\lambda^{-1}z^2-\lambda \bar{z}^2)(1-\frac{r^4}{12})  \\
                          (\lambda^{-1}z^2+\lambda \bar{z}^2)(1-\frac{r^4}{12}) \\
                          \end{array}
                        \right).
\end{equation}
Here $r=|z|$, $\varsigma=\left|1-\frac{r^4}{4}-\frac{2r^6}{9}\right|$.  Moreover, we have
\begin{enumerate}
\item Set $\hat Y= (\hat y_0,\cdots,\hat y_7)^t$. Then $\hat y=\frac{1}{\hat{y}_0}(\hat{y}_1,\cdots,\hat{y}_7)^t=[\hat{Y}]$ is an H-totally isotropic, Willmore immersion from $S^2$ to $S^6$, with metric
\[|\hat y_z|^2|\dd z|^2=\frac{2(1+4r^2+\frac{r^4}{4}+\frac{2r^6}{9}+\frac{4r^8}{9}+\frac{r^{10}}{36}+\frac{r^{12}}{81})}{\left(1+r^2+\frac{5r^4}{4}+\frac{4r^6}{9}+\frac{r^8}{36}\right)^2}|\dd z|^2.
\]
$[\hat{Y}]$ has no dual surface.
\item Set $Y= (y_0,\cdots,y_7)^t$. Then $y=\frac{1}{y_0}(y_1,\cdots,y_7)^t=[Y]$ is an H-totally isotropic, Willmore immersion from $\mathbb{C}$ to $S^6$, with metric
\[|  y_z|^2|\dd z|^2=\frac{2(1+\frac{r^4}{4}+\frac{4r^6}{9})}{\left(1+r^2+\frac{r^4}{4}+\frac{r^6}{9}\right)^2}|\dd z|^2.
\]
Note that $[Y]$ is a Willmore map from $S^2$, with a branched point $z=\infty$. Moreover, $ y$ is conformally equivalent to the minimal surface $ x$ in $\mathbb R^8$:
\[   x= \left(
    \begin{array}{c}
      i (\lambda^{-1}z-\lambda \bar{z})  \\
     - (\lambda^{-1}z+\lambda \bar{z})\\
      -\frac{2i}{r^2}\left(\lambda^{-1}z-\lambda \bar{z} \right) \\
      \frac{2}{r^2}\left(\lambda^{-1}z+\lambda \bar{z} \right) \\
      \frac{2i}{3}(\lambda^{-1}z^2-\lambda \bar{z}^2)\\
      -\frac{2}{3}(\lambda^{-1}z^2+\lambda \bar{z}^2) \\
    \end{array}
  \right). \]
\item The harmonic map $Y\wedge \hat Y$ has no definition on the curve $1-\frac{r^4}{4}-\frac{2r^6}{9}=0$. But the maps $[Y]$ and $[\hat Y]$ are well defined on the whole two-sphere $S^2$.
\end{enumerate}
\end{theorem}

\begin{remark}\
\begin{enumerate}
\item From this we see that it is possible that although the harmonic map $Y\wedge \hat Y$ is not globally well-defined, the Willmore surfaces $[Y]$ and $[\hat Y]$ are well-defined. This is a very interesting phenomena to be explained, which may be related to the Iwasawa decompositions of the loop group $\Lambda G_{\sigma}^{\C}$ of the non-compact group $G=SO^{+}(1,2m+1)$.
\item One can also derive $\hat y$ from a concrete adjoint transform of the minimal surface $x$. To ensure $\hat y $ to be immersed, one need some restrictions on the minimal surface $x$. Our examples here play an important role in the discussions of these conditions. We refer to \cite{MWW2} for more details.
\end{enumerate}\end{remark}

\begin{theorem}\label{thm-example-2} Let
\[\eta=\lambda^{-1}\eta_{-1}\dd z\  \hbox{ with }   \hat B_1=\frac{1}{2}\left(
                                                           \begin{array}{cccccc}
                                                           i &-1 &-i & 1 \\
                                                           i  &-1 & i & -1 \\
                                                           \end{array}
                                                         \right).\]
Then the associated family of the corresponding isotropic harmonic maps is $Y\wedge \hat Y$, with

\begin{equation}\label{eq-Y-ex-2}
  Y=  \frac{\sqrt{2}}{2\varsigma}
 \left(
    \begin{array}{c}
     (1+\frac{r^2}{2})^2\\
      -(1-\frac{r^2}{2})^2\\
i(\lambda^{-1}z-\lambda\bar{z})\\
  -\lambda^{-1}z-\lambda\bar{z}\\
-\frac{ir^2}{2}(\lambda^{-1}z-\lambda\bar{z})\\
      \frac{r^2}{2}(\lambda^{-1}z+\lambda\bar{z})\\
    \end{array}
  \right),\ \hat{Y}= \frac{\sqrt{2}}{2\varsigma}\left(
    \begin{array}{c}
     (1+\frac{r^2}{2})^2\\
       (1-\frac{r^2}{2})^2\\
   \frac{ir^2}{2}(\lambda^{-1}z-\lambda\bar{z})\\
     -\frac{r^2}{2}(\lambda^{-1}z+\lambda\bar{z})\\
    -i(\lambda^{-1}z-\lambda\bar{z})\\
  \lambda^{-1}z+\lambda\bar{z}\\
    \end{array}
  \right).
\end{equation}
Here $r=|z|$ and $\varsigma= 1-\frac{r^4}{4} $. Note that in this case $Y$ is  conformally equivalent to  $\hat Y$. Moreover, $Y$ and $\hat Y$ satisfy  the following results.
\begin{enumerate}
\item Set $Y=(y_0,\cdots,y_5)^t$. Then $y=\frac{1}{y_0}(y_1,\cdots,y_5)^t=[Y]$ is an H-totally isotropic, Willmore immersion from $S^2$ to $S^4$, with metric
$\langle y_{z},y_{\bar{z}}\rangle|\dd z|^2=\frac{2+\frac{r^4}{2}}{(1+\frac{r^2}{2})^4}|\dd z|^2.$
\item  $[Y]$ is conformally equivalent to the minimal surface $x$ in $\mathbb R^4$:
\[x= \left(
       \begin{array}{cccc}
     \frac{2i \lambda^{-1}}{\bar z}-\frac{2i\lambda }{z}  &
        -\frac{2 \lambda^{-1}}{\bar z}-\frac{2\lambda }{z} &  -i\lambda^{-1}z+i\lambda\bar{z}  &  \lambda^{-1}z+\lambda\bar{z} \\
       \end{array}
     \right)^t.
\]
\item  The harmonic map  $Y\wedge \hat Y$ has no definition on the curve $1-\frac{r^4}{4}=0$. But both $[Y]$ and $[\hat Y]$ are   Willmore immersions on the whole two-sphere $S^2$.
\end{enumerate}
\end{theorem}

\begin{remark}Set $\lambda=1$. On the curve $\Gamma:\ 1-\frac{r^4}{4}=0$ we have $z= \sqrt{2}e^{i\theta}$ and hence
\[\begin{split}(1-\frac{r^4}{4})Y|_{\Gamma}&=\left(
                       \begin{array}{cccccc}
                         2\sqrt{2} & 0 & -2\sin\theta & -2\cos\theta & 2\sin\theta & 2\cos\theta  \\
                       \end{array}
                     \right),\\
                     (1-\frac{r^4}{4})\hat Y|_{\Gamma}&=\left(
                       \begin{array}{cccccc}
                         2\sqrt{2} & 0 & -2\sin\theta & -2\cos\theta & 2\sin\theta & 2\cos\theta  \\
                       \end{array}
                     \right).
                     \end{split}
\]
So $Y\wedge \hat Y$ has no definition on the curve $\Gamma$.
\end{remark}
\section{Appendix: Iwasawa decompositions and computations of examples }
This section contains two parts: the proof of the technical lemmas in Section 5.1 and the computations of the examples in Section 5.2.
\subsection{On the technical lemmas of Section 5.1}

To begin with, it is convenient to have the explicit expressions of $\mathcal{P}(A)$ in \eqref{eq-Lie}. So we will first give this expression and then provide the proofs of  Lemma \ref{lemma-gm-2} and  Lemma \ref{lemma-gm-3}.

\subsubsection{On $\mathcal{P}(A)$}
Set
\[A=(\mathbf{a}_{ij}),~ \mathcal{P}(A)=B=(\mathbf{b}_{ij}), ~ \hat j=2m+3- j ~ \hbox{ and } \hat k=2m+3- k.\]
Then when $ j=1,\cdots, m,$ we have
\begin{equation}\label{eq-T1}
\mathbf{b}_{jk}=\left\{\begin{split}
&\frac{\mathbf{a}_{2j+1,2k+1}-i\mathbf{a}_{2j+2,2k+1}+i\mathbf{a}_{2j+1,2k+2}+\mathbf{a}_{2j+2,2k+2}}{2},~ k=1,\cdots, m;\\
&\frac{i\mathbf{a}_{2j+1,1}+\mathbf{a}_{2j+2,1}+i\mathbf{a}_{2j+1,2}+\mathbf{a}_{2j+2,2}}{2},~~ k=m+1;\\
&\frac{-i\mathbf{a}_{2j+1,1}-\mathbf{a}_{2j+2,1}+i\mathbf{a}_{2j+1,2}+\mathbf{a}_{2j+2,2}}{2},~~ k=m+2;\\
&\frac{-\mathbf{a}_{2j+1,2\hat{k}+1}+i\mathbf{a}_{2j+2,2\hat{k}+1}+i\mathbf{a}_{2 j+1,2\hat{k}+2}+\mathbf{a}_{2j+2,2\hat{k}+2}}{2},~ k=m+3,\cdots, 2m+2.\\
\end{split}\right.
\end{equation}
When $ j=m+1$ we have
\begin{equation}\label{eq-T2}
\mathbf{b}_{jk}=\left\{\begin{split}
&\frac{-i\mathbf{a}_{1,2k+1}-i\mathbf{a}_{2,2k+1}+\mathbf{a}_{1,2k+2}+\mathbf{a}_{2,2k+2}}{2},~~k=1,\cdots, m;\\
&\frac{\mathbf{a}_{11}+\mathbf{a}_{21}+\mathbf{a}_{12}+\mathbf{a}_{22}}{2},~~ k=m+1;\\
&\frac{-\mathbf{a}_{11}-\mathbf{a}_{21}+\mathbf{a}_{12}+\mathbf{a}_{22}}{2},~~ k=m+2;\\
&\frac{i\mathbf{a}_{1,2\hat k+1}+i\mathbf{a}_{2,2\hat k+1}+\mathbf{a}_{1,2\hat k+2}+\mathbf{a}_{2,2\hat k+2}}{2},~~k=m+3,\cdots, 2m+2.\\
\end{split}\right.
\end{equation}
When $ j=m+2$ we have
\begin{equation}\label{eq-T3}
\mathbf{b}_{jk}=\left\{\begin{split}
&\frac{i\mathbf{a}_{1,2k+1}-i\mathbf{a}_{2,2k+1}-\mathbf{a}_{1,2k+2}+\mathbf{a}_{2,2k+2}}{2},~~k=1,\cdots, m;\\
&\frac{-\mathbf{a}_{11}+\mathbf{a}_{21}-\mathbf{a}_{12}+\mathbf{a}_{22}}{2},~~ k=m+1;\\
&\frac{\mathbf{a}_{11}-\mathbf{a}_{21}-\mathbf{a}_{12}+\mathbf{a}_{22}}{2},~~ k=m+2;\\
&\frac{-i\mathbf{a}_{1,2\hat k+1}+i\mathbf{a}_{2,2\hat k+1}-\mathbf{a}_{1,2\hat k+2}+\mathbf{a}_{2,2\hat k+2}}{2},~~k=m+3,\cdots, 2m+2.\\
\end{split}\right.
\end{equation}
When $ j=m+3,\cdots, 2m+2$, we have
\begin{equation}\label{eq-T4}
\mathbf{b}_{jk}=\left\{\begin{split}
&\frac{-\mathbf{a}_{2\hat j+1,2k+1}-i\mathbf{a}_{2\hat j+2,2k+1}-i\mathbf{a}_{2\hat j+1,2k+2}+\mathbf{a}_{2\hat j+2,2k+2}}{2},~ k=1,\cdots, m;\\
&\frac{-i\mathbf{a}_{2\hat j+1,1}+\mathbf{a}_{2\hat j+2,1}-i\mathbf{a}_{2\hat j+1,2}+\mathbf{a}_{2\hat j+2,2}}{2},~~ k=m+1;\\
&\frac{i\mathbf{a}_{2\hat j+1,1}-\mathbf{a}_{2\hat j+2,1}-i\mathbf{a}_{2\hat j+1,2}+\mathbf{a}_{2\hat j+2,2}}{2},~~ k=m+2;\\
&\frac{\mathbf{a}_{2\hat{j}+1,2\hat{k}+1}+i\mathbf{a}_{2\hat{j}+2,2\hat{k}+1}-i\mathbf{a}_{2\hat{j}+1,2\hat{k}+2}
+\mathbf{a}_{2\hat{j}+2,2\hat{k}+2}}{2},~ k=m+3,\cdots, 2m+2.\\
\end{split}\right.
\end{equation}

\subsubsection{Proof of Lemma \ref{lemma-gm-2}}\

(1).
Assume that $\eta_{-1}=(\mathrm{a}_{jk})$. Then one has
\[\mathbf{a}_{1,2j+1}=\mathbf{a}_{2j+1,1}=h_{j1},~\mathbf{a}_{1,2j+2}=\mathbf{a}_{2j+2,1}=ih_{j1},~~\]
\[\mathbf{a}_{2,2j+1}=-\mathbf{a}_{2j+1,2}=\hat h_{j1},~\mathbf{a}_{2,2j+2}=-\mathbf{a}_{2j+2,2}=i\hat h_{j1},\]
when $j=1,\cdots, m+1$, and $\mathbf{a}_{jk}=0$ otherwise. Substituting into \eqref{eq-T1}--\eqref{eq-T4}, one obtains \eqref{eq-P-eta}.

(2). First by definition, $H(0,0,\lambda)=I$. Next,
\[H_z=\lambda^{-1}H_{1z}+\lambda^{-2}H_{2z}=\left(
        \begin{array}{ccc}
          0 & \lambda^{-1} \check f & -\lambda^{-2} f{\check{f}^{\sharp}} \\
          0 & 0 & -\lambda^{-1}\check f ^{\sharp} \\
          0 & 0 & 0 \\
        \end{array}
      \right)=H \mathcal{P}(\eta_{-1}).
\]

(3).  First since $H(0,0,\lambda)=I$, when $|z|<\varepsilon$, there exists an Iwasawa Decomposition $H=\tilde F V_+$, with $\tilde F\in\Lambda G_{\sigma}$, $V_+\in\Lambda^+ G_{\sigma}^{\C}$.
Next we want to express $\tilde F$ in terms of $H$.
Since $H=I+\lambda^{-1}H_1+\lambda^{-2}H_2$, by the reality condition we see that $V_+=V_0+\lambda V_1+\lambda^2 V_2$ with $V_0$, $V_1$ and $V_2$ independent of $\lambda$.

Assume that $V_+=V_0\hat V_+$ such that $\hat V_+|_{\lambda=0}=I$. Then we have
\[\tilde F=H\hat  V_+^{-1}V_0^{-1}.\]
Since $\tau(\tilde F)=\tilde F$, we obtain $\tau(H)\tau(\hat V_+^{-1})\tau(V_0^{-1})=H\hat  V_+^{-1}V_0^{-1}$, i.e.,
\[\tau(H)^{-1}H=\tau(\hat V_+^{-1})\tau(V_0^{-1})V_0\hat  V_+.\]
We then assume that
\[\tau(H)^{-1}H=WW_0\tau(W)^{-1}\]
with $W_0=\tau(V_0^{-1})V_0$, $W=I+\lambda^{-1}W_1+\lambda^{-2}W_2=\tau(\hat V_+^{-1})$ and
\[W_0=\left(
        \begin{array}{ccc}
          a & 0 & b \\
          0 & q & 0 \\
          c & 0 & \varrho \\
        \end{array}
      \right),\  W_1=\left(
        \begin{array}{ccc}
          0 & u & 0 \\
         -u_0^{\sharp} & 0 & -u^{\sharp} \\
          0 &  u_0 & 0 \\
        \end{array}
      \right).\]
Comparing the coefficients of $\lambda$, we obtain
\begin{equation}\label{eq-compare}\left\{
\begin{split}
&W_2W_0=H_2,\\
&W_1W_0+W_2W_0\tau(W_1)=H_1+\tau(H_1)H_2,\\
&W_0+W_1W_0\tau(W_1)+W_2W_0\tau(W_2)=I+\tau(H_1)H_1+\tau(H_2)H_2.
\end{split}
\right.
\end{equation}
  Direct computation shows that
\[\tau(H)^{-1}=\left(
        \begin{array}{ccc}
          I & 0& 0\\
          \lambda J\bar{f}^t  & I & 0 \\
          \lambda^{2}J\bar{g}^t  & -\lambda \bar{f}^{\sharp,t}J & I \\
        \end{array}
      \right),\]
and
\begin{equation}\label{eq-H}
\tau(H)^{-1}H=\left(
        \begin{array}{ccc}
         I & \lambda^{-1}f & \lambda^{-2}g \\
          \lambda J\bar{f}^t  & I+J\bar{f}^tf & \lambda^{-1}(J\bar{f}^t-f^{\sharp}) \\
          \lambda^{2}J\bar{g}^t  & -\lambda( \bar{f}^{\sharp,t}J-\bar{g}^tf) & I+\bar{f}^{\sharp,t}Jf^{\sharp}+\bar{g}^tg \\
        \end{array}
      \right).
\end{equation}
From the first two equations of \eqref{eq-compare} we can see that
\[\begin{split}W_1W_0&=H_1+\tau(H_1)H_2-W_2W_0\tau(W_1)=H_1+\tau(H_1)H_2-H_2\tau(W_1)\\
&=\left(
    \begin{array}{ccc}
      0 & \cdots &  0\\
      0 & 0 &  \cdots \\
      0 & 0 & 0 \\
    \end{array}
  \right)+\left(
    \begin{array}{ccc}
      0 & 0 &  0\\
      0 & 0 &  \cdots \\
      0 & 0 & 0 \\
    \end{array}
  \right)-\left(
    \begin{array}{ccc}
      0 & \cdots &  0\\
      0 & 0 & 0 \\
      0 & 0 & 0 \\
    \end{array}
  \right).
\\
\end{split}\]
So
\[W_1=\left(
    \begin{array}{ccc}
      0 & \cdots &  0\\
      \cdots & 0 &  \cdots \\
      0 & 0 & 0 \\
    \end{array}
  \right), \hbox{ i.e., $u_0=0$. }\]
Then from the last equation of \eqref{eq-compare} we see that
\[\begin{split}W_0&=\left(I+\tau(H_1)H_1+\tau(H_2)H_2\right)-W_1W_0\tau(W_1)-W_2W_0\tau(W_2)\\
&=\left(
    \begin{array}{ccc}
     \cdots  &0 &  0\\
      0 &  \cdots & 0\\
      0 & 0 & \cdots \\
    \end{array}
  \right)-\left(
    \begin{array}{ccc}
      \cdots  & 0 &  0\\
      0 & \cdots  & 0 \\
      0 & 0 & 0 \\
    \end{array}
  \right)-\left(
    \begin{array}{ccc}
      \cdots  & 0 &  0\\
      0 & 0 & 0 \\
      0 & 0 & 0 \\
    \end{array}
  \right),
\\
\end{split}\]
i.e., $b=c=0$. So we have that
\[W_0=\left(
        \begin{array}{ccc}
          a & 0 & 0\\
          0 & q & 0 \\
          0 & 0 & \varrho \\
        \end{array}
      \right),\  W=\left(
        \begin{array}{ccc}
          I & \lambda^{-1}u & \lambda^{-2}v \\
         0 & I & -\lambda^{-1}u^{\sharp} \\
          0 &  0 &  I  \\
        \end{array}
      \right) \hbox{ and }\tau(W)^{-1}=\left(
        \begin{array}{ccc}
          I & 0& 0\\
          \lambda J\bar{u}^t  & I & 0 \\
          \lambda^{2}J\bar{v}^t  & -\lambda \bar{u}^{\sharp,t}J & I \\
        \end{array}
      \right).\]
So            \begin{equation}\label{eq-W}
 WW_0\tau(W)^{-1}=\left(
        \begin{array}{ccc}
         a+uqJ\bar{u}^t+v \varrho\bar{v}^t & \lambda^{-1}(uq-v\varrho\bar{u}^{\sharp,t}J) & \lambda^{-2}v\varrho \\
          \lambda (qJ\bar{u}^t-u^{\sharp}\varrho\bar{v}^t)  & q+u^{\sharp}\varrho\bar{u}^{\sharp,t}J & -\lambda^{-1}u^{\sharp}\varrho \\
          \lambda^{2}\varrho\bar{v}^t  & -\lambda \varrho \bar{u}^{\sharp,t} & \varrho \\
        \end{array}
      \right).
      \end{equation}
By \eqref{eq-H}, \eqref{eq-W} and $\tau(H)^{-1}H=WW_0\tau(W)^{-1}$ we obtain
\eqref{eq-Iwasawa:1}.

Apparently $W_0$ has the decomposition \[W_0=\tau(L)^{-1}L,\ \hbox{ with } L=\hbox{diag}\{l_1,l_0,l_4\}.\]
Now set $\tilde{F}=H\tau(W)L^{-1}$. We see that $\tau(\tilde F)=\tilde F$ and
$H=\tilde{F} L \tau(W)^{-1}$ is an Iwasawa decomposition of $H$.
Substituting
\[\tau(W)=\left(
        \begin{array}{ccc}
          I & 0& 0\\
          -\lambda \bar{u}^{\sharp}J  & I & 0 \\
          \lambda^{2}J\bar{v}J  & \lambda J\bar{u} & I \\
        \end{array}
      \right),\]
$L^{-1}$      and $H$ into  $\tilde{F}=H\tau(W)L^{-1}$, one obtains \eqref{eq-Iw}.

(4).  Since
\[\tau(W)= I+\lambda \tau (W_1)+\lambda^2\tau (W_2),\ \tau(W)^{-1}=I-\lambda \tau(W_1)+\lambda^2(\cdots).\]
So
\[\tilde{\alpha}_1'=L\mathcal{P}(\eta_{-1})L^{-1}\dd z,~~~ \tilde{\alpha}_0¡¯=\left(-\tau(W_1)\mathcal{P}(\eta_{-1})+\mathcal{P}(\eta_{-1})\tau(W_1)-L_zL^{-1}\right)\dd z.\]
Then \eqref{eq-alpha-1} and \eqref{eq-alpha-2} follow. This finishes the proof of Lemma  \ref{lemma-gm-2}.

\subsubsection{Proof of Lemma \ref{lemma-gm-3}}
Assume that \[\hbox{$\tilde{F}=\mathcal{P}(F)$ and $\alpha'=F^{-1}F_z\dd z=(\mathbf{b}_{jk})\dd z$.}\] So we have $\mathcal{P}(\alpha')= \tilde{\alpha}_1'+\tilde{\alpha}_0'$. Applying \eqref{eq-T1}--\eqref{eq-T4}, \eqref{eq-alpha-1} and \eqref{eq-alpha-2}, we obtain that
\[ \begin{split}
&\frac{-\mathbf{b}_{2j+1,2\hat{k}+1}+i\mathbf{b}_{2j+2,2\hat{k}+1}+i\mathbf{b}_{2 j+1,2\hat{k}+2}+\mathbf{b}_{2j+2,2\hat{k}+2}}{2}=0,~ 1\leq j\leq m,~ m+3\leq k\leq2m+2;\\
&\frac{-i\mathbf{b}_{1,2k+1}-i\mathbf{b}_{2,2k+1}+\mathbf{b}_{1,2k+2}+\mathbf{b}_{2,2k+2}}{2}=0,~~ 1\leq k\leq m;\\
&\frac{i\mathbf{b}_{1,2k+1}-i\mathbf{b}_{2,2k+1}-\mathbf{b}_{1,2k+2}+\mathbf{b}_{2,2k+2}}{2}=0,~~ 1\leq k\leq m;\\
&\frac{-\mathbf{b}_{2\hat j+1,2k+1}-i\mathbf{b}_{2\hat j+2,2k+1}-i\mathbf{b}_{2\hat j+1,2k+2}+\mathbf{b}_{2\hat j+2,2k+2}}{2}=0,~~ m+3\leq j\leq2m+2, 1\leq k\leq m.\\
\end{split}
\]
Here $\hat j=2m+3-j$, $\hat k=2m+3-k$.  Since $\mathbf{b}_{jk}=-\mathbf{b}_{kj}$ for $j,k>1$, and $\mathbf{b}_{1k}=\mathbf{b}_{k1}$ for $k>1$, we have
\[\left\{\begin{split}
&\mathbf{b}_{j,2k+2}=i\mathbf{b}_{j,2k+1},~ j=1,2,~~1\leq k\leq m;\\
&\mathbf{b}_{2j+1,2k+1}=\mathbf{b}_{2j+2,2 k+2},~~1\leq j, k\leq m;\\
&\mathbf{b}_{2j+1,2k+2}=-\mathbf{b}_{2j+2,2 k+1},~~1\leq j, k\leq m.
\end{split}\right.\]
Set \[F=(e_0,\hat e_0, \psi_1,\hat\psi_1, \cdots,\psi_m,\hat\psi_m) ,~~Y=\frac{\sqrt{2}}{2}(e_0-\hat e_0)~ \hbox{ and }~\hat Y=\frac{\sqrt{2}}{2}(e_0+\hat e_0).\] We have then
\[\left\{\begin{split}
&e_{0z}=\mathbf{b}_{21}\hat e_0+\sum_{1\leq j\leq m}\mathbf{b}_{1,2j+1}(\psi_j+i\hat\psi_j),\\
&\hat e_{0z}=\mathbf{b}_{21} e_0-\sum_{1\leq j\leq m}\mathbf{b}_{2,2j+1}(\psi_j+i\hat\psi_j),\\
&\psi_{jz}=\mathbf{b}_{1,2j+1}e_0+\mathbf{b}_{2,2j+1}\hat e_0-\sum_{1\leq k\leq m}\left(\mathbf{b}_{2j+1,2k+1}\psi_k+\mathbf{b}_{2j+1,2k+2}\hat\psi_k\right),\\
&\hat\psi_{jz}=\mathbf{b}_{1,2j+2}e_0+\mathbf{b}_{2,2j+2}\hat e_0-\sum_{1\leq k\leq m}\left(\mathbf{b}_{2j+2,2k+1}\psi_k+\mathbf{b}_{2j+2,2k+2}\hat\psi_k\right).\\
\end{split}\right.\]
So
\[(\psi_j+i\hat\psi_j)_{z}=-\sum_{1\leq k\leq m}\left(\mathbf{b}_{2j+1,2k+1}+i\mathbf{b}_{2j+2,2k+1}\right)\left(\psi_k+i\hat\psi_k\right)\mod \{Y,\hat Y\}.\]
\[(\psi_j+i\hat\psi_j)_{\bar{z}}=-\sum_{1\leq k\leq m}\left(\overline{\mathbf{b}_{2j+1,2k+1}}+i\overline{\mathbf{b}_{2j+2,2k+1}}\right)\left(\psi_k+i\hat\psi_k\right)\mod \{Y,\hat Y\}.\]
Since
\[Y_z=-\mathbf{b}_{21}Y+\frac{\sqrt{2}}{2}\sum_j(\mathbf{b}_{1,2j+1}+\mathbf{b}_{2,2j+1})(\psi_j+i\hat\psi_j)\]
and
\[\hat Y_z= \mathbf{b}_{21}\hat Y+\frac{\sqrt{2}}{2}\sum_j(\mathbf{b}_{1,2j+1}-\mathbf{b}_{2,2j+1})(\psi_j+i\hat\psi_j),\]
it is straightforward to verify that $Y$ and $\hat Y$ satisfy \eqref{eq-HTI-0} and \eqref{eq-HTI}. This finishes the proof of Lemma  \ref{lemma-gm-3}.

\subsection{Computations on the examples}

This subsection is to derive the examples stated in Section 5.  To begin with, first we recall the formula of expressing $Y$ and $\hat Y$ by elements of $H$. Then we will apply the formula to derive the examples.

     \subsubsection{From frame to Willmore surfaces}

Suppose that $F=(e_0,\hat{e}_0,\psi_1,\hat\psi_1,\cdots,\psi_m,\hat\psi_m)$,  and
\[Y=\frac{\sqrt{2}}{2}(e_0-\hat{e}_0),\ \hat{Y}=\frac{\sqrt{2}}{2}(e_0+\hat{e}_0)\]
Then $y=[Y]$ and $\hat y=[\hat Y]$are the two Willmore surfaces (which may have branched points) adjoint to each other.

Assume that  $F=(\mathbf{c}_{jk})$, $\tilde F=(\widetilde{\mathbf{c}}_{jk})$. Then by \eqref{eq-T1}--\eqref{eq-T4}, we have that
\[\left\{\begin{split}&\widetilde{\mathbf{c}}_{j,m+1}+\widetilde{\mathbf{c}}_{\hat{j},m+1}=\mathbf{c}_{2j+2,1}+\mathbf{c}_{2j+2,2},\\
& -i(\widetilde{\mathbf{c}}_{j,m+1}-\widetilde{\mathbf{c}}_{\hat{j},m+1})=\mathbf{c}_{2j+1,1}+\mathbf{c}_{2j+1,2},\\
 &\widetilde{\mathbf{c}}_{m+1,m+1}+\widetilde{\mathbf{c}}_{m+2,m+1}=\mathbf{c}_{21}+\mathbf{c}_{22},\\
&\widetilde{\mathbf{c}}_{m+1,m+1}-\widetilde{\mathbf{c}}_{m+2,m+1}=\mathbf{c}_{11}+\mathbf{c}_{12},\end{split}\right.\
\left\{\begin{split}&\widetilde{\mathbf{c}}_{j,m+2}+\widetilde{\mathbf{c}}_{\hat{j},m+2}=-\mathbf{c}_{2j+2,1}+\mathbf{c}_{2j+2,2},\\
&-i(\widetilde{\mathbf{c}}_{j,m+2}-\widetilde{\mathbf{c}}_{\hat{j},m+2})=-\mathbf{c}_{2j+1,1}+\mathbf{c}_{2j+1,2},\\
&\widetilde{\mathbf{c}}_{m+1,m+2}+\widetilde{\mathbf{c}}_{m+2,m+2}=-\mathbf{c}_{21}+\mathbf{c}_{22},\\
 &\widetilde{\mathbf{c}}_{m+1,m+2}-\widetilde{\mathbf{c}}_{m+2,m+2}=-\mathbf{c}_{11}+\mathbf{c}_{12}.
\end{split}\right.\]
So we have
                      \begin{equation}\label{eq-YandY}
                         Y=-\frac{\sqrt{2}}{2}\left(
                        \begin{array}{c}
                          \widetilde{\mathbf{c}}_{m+1,m+2}-\widetilde{\mathbf{c}}_{m+2,m+2} \\
                          \widetilde{\mathbf{c}}_{m+1,m+2}+\widetilde{\mathbf{c}}_{m+2,m+2} \\
                         -i(\widetilde{\mathbf{c}}_{1,m+2}-\widetilde{\mathbf{c}}_{2m+2,m+2})\\
                          \widetilde{\mathbf{c}}_{1,m+2}+\widetilde{\mathbf{c}}_{2m+2,m+2} \\
                          \cdots \\
                         -i(\widetilde{\mathbf{c}}_{m,m+2}-\widetilde{\mathbf{c}}_{m+3,m+2})\\
                          \widetilde{\mathbf{c}}_{m,m+2}+\widetilde{\mathbf{c}}_{m+3,m+2} \\
                        \end{array}
                      \right),\ \hat Y= \frac{\sqrt{2}}{2}\left(
                        \begin{array}{c}
                          \widetilde{\mathbf{c}}_{m+1,m+1}-\widetilde{\mathbf{c}}_{m+2,m+1} \\
                          \widetilde{\mathbf{c}}_{m+1,m+1}+\widetilde{\mathbf{c}}_{m+2,m+1}\\
                         -i(\widetilde{\mathbf{c}}_{1,m+1}-\widetilde{\mathbf{c}}_{2m+2,m+1})\\
                          \widetilde{\mathbf{c}}_{1,m+21}+\widetilde{\mathbf{c}}_{2m+2,m+1} \\
                          \cdots \\
                         -i(\widetilde{\mathbf{c}}_{m,m+1}-\widetilde{\mathbf{c}}_{m+3,m+1})\\
                          \widetilde{\mathbf{c}}_{m,m+1}+\widetilde{\mathbf{c}}_{m+3,m+1} \\
                        \end{array}
                      \right).
                      \end{equation}

\subsubsection{Proof of Theorem \ref{thm-example-1}}
Set $r=|z|$. By \eqref{eq-P-eta} we have
 \[\check{f}=\left(
      \begin{array}{cc}
        1 & 0 \\
        0 & 1 \\
        2z & 0 \\
      \end{array}
    \right).\]
By integration one has
\[ f=\left(
      \begin{array}{cc}
        z & 0 \\
        0 & z \\
        z^2 & 0 \\
      \end{array}
    \right)~~\hbox{ and }~~
g=-\left(
                             \begin{array}{ccc}
                               0 & \frac{z^2}{2}  &0 \\
                               \frac{2z^3}{3} & 0&  \frac{z^2}{2}  \\
                              0 &  \frac{z^3}{3}&0  \\
                             \end{array}
                           \right).\]
Substituting into \eqref{eq-Iwasawa:1F}, one obtains
\[\varrho=\left(
                                   \begin{array}{ccc}
                                    1+ \frac{4r^6}{9} &r^2\bar{z} & \frac{r^4\bar{z}}{3} \\
                                    r^2z & 1+\frac{r^4}{4}+\frac{r^6}{9} & r^2 \\
                                     \frac{r^4z}{3} & r^2 & 1+\frac{r^4}{4}\\
                                   \end{array}
                                 \right),
\]
with
\[\varrho^{-1}=\frac{1}{|\varrho|}\left(
                                   \begin{array}{ccc}
                                    (1-\frac{r^4}{4})^2+ \frac{4r^6}{9}(1+\frac{r^4}{4}) & -\bar{z}r^2(1-\frac{r^4}{12}) & \frac{2r^4\bar{z}}{3}(1-\frac{r^4}{8}-\frac{r^6}{18}) \\
                                    -zr^2(1-\frac{r^4}{12}) & 1+\frac{r^4}{4}+\frac{4r^6}{9} & -r^2(1+\frac{r^6}{9}) \\
                                     \frac{2r^4z}{3}(1-\frac{r^4}{8}-\frac{r^6}{18}) &  -r^2(1+\frac{r^6}{9}) &  (1-\frac{2r^6}{9})^2+\frac{r^4}{4}(1+\frac{4r^6}{9})\\
                                   \end{array}
                                 \right).\]
Here $|\varrho|=\varsigma^2$ and $\varsigma=1-\frac{r^4}{4}-\frac{2r^6}{9}$.  Then by \eqref{eq-Iwasawa:1E} one obtains
                         \[\begin{split}u^{\sharp}& =(f^{\sharp}-Jf^tg)\varrho^{-1}\\
&=\frac{z}{\varsigma }\left(
                                                                \begin{array}{ccc}
                                                                 -\frac{r^2z}{3}  & 1 & -\frac{r^2}{2}\\
                                                                 z(1-\frac{r^4}{12})  & -\frac{r^2}{2}-\frac{2r^4}{3} & 1+\frac{r^6}{9}\\
                                                                \end{array}
                                                              \right).\\
                              \end{split}\]
 So
       \[u=\frac{z}{\varsigma}\left(
                                                                \begin{array}{ccc}
                                                                 1+\frac{r^6}{9}  & -\frac{r^2}{2}\\
                                                                  -\frac{r^2}{2}-\frac{2r^4}{3}  &  1\\
                                                                 z(1-\frac{r^4}{12})  & -\frac{r^2z}{3}  \\
                                                                \end{array}
                                                              \right).\]
Substituting $f$, $g$ and $u$ into \eqref{eq-Iwasawa:1D}, one obtains
\[q=I_2.\]
So $l_0=I_2$.
By \eqref{eq-Iw}, since
\[   \begin{split} J\bar{u}&=\frac{\bar{z}}{\varsigma}\left(
                                                                \begin{array}{ccc}
                                                                 \bar{z}(1-\frac{r^4}{12})  & -\frac{r^2\bar{z}}{3}  \\
                                                                  -\frac{r^2}{2}-\frac{2r^4}{3}  &  1\\
                                                                 1+\frac{r^6}{9}  & -\frac{r^2}{2}\\
                                                                \end{array}
                                                              \right),\\
 f+gJ\bar{u}&=\frac{z}{\varsigma}\left(
                                                                \begin{array}{ccc}
                                                                 1+\frac{r^6}{9}  & -\frac{r^2}{2}  \\
                                                                  -\frac{r^2}{2}(1+\frac{4r^2}{3})  &  1\\
                                                                 z(1-\frac{r^4}{12})  & -\frac{r^2z}{3}\\
                                                                \end{array}
                                                              \right),\\
I-f^{\sharp}J\bar{u}&=\frac{1}{\varsigma}\left(
                                                                \begin{array}{cc}
                                                                 1+\frac{r^4}{4}+\frac{4r^6}{9}  & -r^2  \\
                                                                  -r^2(1+r^2+\frac{r^6}{36})  &  1+\frac{r^4}{4}+\frac{r^6}{9}\\
                                                                \end{array}
                                                              \right),
                                                              \end{split}\]
we have
\[\left(
    \begin{array}{cc}
      \widetilde{\mathbf{c}}_{14} & \widetilde{\mathbf{c}}_{15} \\
      \widetilde{\mathbf{c}}_{24} & \widetilde{\mathbf{c}}_{25} \\
      \widetilde{\mathbf{c}}_{34} & \widetilde{\mathbf{c}}_{35} \\
      \widetilde{\mathbf{c}}_{44} & \widetilde{\mathbf{c}}_{45} \\
      \widetilde{\mathbf{c}}_{54} & \widetilde{\mathbf{c}}_{55} \\
      \widetilde{\mathbf{c}}_{64} & \widetilde{\mathbf{c}}_{65} \\
      \widetilde{\mathbf{c}}_{74} & \widetilde{\mathbf{c}}_{75} \\
      \widetilde{\mathbf{c}}_{84} & \widetilde{\mathbf{c}}_{85} \\
    \end{array}
  \right)=\frac{1}{\varsigma}\left(
    \begin{array}{cc}
    \lambda^{-1}z(1+\frac{r^6}{9})  & -\lambda^{-1}\frac{zr^2}{2}  \\
     -\lambda^{-1}\frac{r^2z}{2}(1+\frac{4r^2}{3})  & \lambda^{-1}z \\
      \lambda^{-1}z^2(1-\frac{r^4}{12})  & -\lambda^{-1}\frac{r^2z^2}{3}\\
       1+\frac{r^4}{4}+\frac{4r^6}{9}  & -r^2  \\
        -r^2(1+r^2+\frac{r^6}{36})  &  1+\frac{r^4}{4}+\frac{r^6}{9}\\
         \lambda\bar{z}^2(1-\frac{r^4}{12})  & -\lambda\frac{r^2\bar{z}^2}{3}  \\
          -\lambda\frac{r^2\bar{z}}{2}(1+\frac{4r^2}{3})  &  \lambda\bar{z}\\
           \lambda\bar{z}(1+\frac{r^6}{9})  & - \lambda\frac{r^2\bar{z}}{2}\\
    \end{array}
  \right).
\]
Substituting $ \widetilde{\mathbf{c}}_{jk}$ into \eqref{eq-YandY}, one derives \eqref{eq-Y-ex-1}.

The rest are straightforward computations, except $\hat y$ being branched at $z=\infty$ and $\hat y$ being unbranched at $z=\infty$.
To this end, we need to use another coordinate.
Set $\tilde{z}=\frac{1}{z}$ and $\tilde{r}=\sqrt{|\tilde{z}|}$, we have that
\[|y_{\tilde{z}}|^2|\dd\tilde{z}|^2=\frac{2\tilde{r}^2(\tilde{r}^6+\frac{\tilde{r}^2}{4}+\frac{4}{9})}{\left(\tilde{r}^6+\tilde{r}^4+\frac{\tilde{r}^2}{4}+\frac{1}{9}\right)^2}|\dd\tilde{z}|^2,
\]
\[|\hat y_{\tilde{z}}|^2|\dd\tilde{z}|^2=
\frac{2\tilde{r}^{12}+8\tilde{r}^{10}+\frac{\tilde{r}^8}{2}+\frac{4\tilde{r}^6}{9}+\frac{8\tilde{r}^4}{9}+\frac{\tilde{r}^{2}}{18} +\frac{2}{81}}{\left(\tilde{r}^8+\tilde{r}^6+\frac{5\tilde{r}^4}{4}+\frac{4\tilde{r}^2}{9}+\frac{1}{36}\right)^2}|\dd\tilde{z}|^2.\]
At $\tilde{z}=0$, $y_{\tilde{z}}|^2|\dd\tilde{z}|^2=0$ and $|\hat y_{\tilde{z}}|^2|\dd\tilde{z}|^2=32|\dd\tilde{z}|^2.$ This finishes the proof.

\subsubsection{Proof of Theorem \ref{thm-example-2}}
Set $r=|z|$. By \eqref{eq-P-eta} we have
\[\check f=\left(
      \begin{array}{cc}
        0 & 1 \\
        1 & 0 \\
      \end{array}
    \right)~~ \Rightarrow ~~f=\left(
      \begin{array}{cc}
        0 & z \\
        z & 0 \\
      \end{array}
    \right), \ g=-\frac{z^2}{2}\left(
                             \begin{array}{ccc}
                                1 & 0 \\
                                 0& 1  \\
                               \end{array}
                           \right).
\]
Substituting into \eqref{eq-Iwasawa:1F}, one obtains
\[\varrho=\left(
      \begin{array}{ccc}
        1+\frac{r^4}{4}& r^2\\
         r^2 & 1+\frac{r^4}{4}\\
      \end{array}
    \right)\ \hbox{ and }\
\ \varrho^{-1}=\frac{1}{|\varrho|}\left(
      \begin{array}{ccc}
       1+\frac{r^4}{4}& -r^2\\
        -r^2 & 1+\frac{r^4}{4}\\
      \end{array}
    \right).\]
    Here $|\varrho|=\varsigma^2$ with $\varsigma=  1-\frac{r^4}{4}$. Next one computes
\[u^{\sharp}=\frac{z}{\varsigma}\left(
                         \begin{array}{ccc}
                            -\frac{r^2}{2} & 1 \\
                          1 & -\frac{r^2}{2}\\
                         \end{array}
                       \right), \ u=\frac{z}{\varsigma}\left(
      \begin{array}{cc}
        -\frac{r^2}{2} & 1 \\
        1 &  -\frac{r^2}{2} \\
      \end{array}
    \right)~~
\hbox{ and } ~~q=I_2.\]
So $l_0=I_2$. Then we have
\[ \lambda J \bar{u}l_0^{-1}=\frac{ \lambda\bar{z}}{\varsigma}\left(
      \begin{array}{cc}
      1&  -\frac{r^2}{2}  \\
  -\frac{r^2}{2} &1 \\
      \end{array}
    \right),~~
    \lambda^{-1}(f+gJ\bar{u})l_0^{-1}=\frac{\lambda^{-1}z}{\varsigma} \left(
      \begin{array}{cc}
        -\frac{r^2}{2} & 1\\
        1 & -\frac{ r^2}{2 }  \\
      \end{array}
    \right)\]and
\[
(I-f^{\sharp}J\bar{u})l_0^{-1}=\frac{1}{\varsigma}\left(
                                                       \begin{array}{cc}
                                                         1+\frac{r^4}{4} & -r^2 \\
                                                         -r^2 & 1+\frac{r^4}{4} \\
                                                       \end{array}
                                                     \right).
\]
By \eqref{eq-Iw},  we have
\[\left(
    \begin{array}{cc}
      \widetilde{\mathbf{c}}_{13} & \widetilde{\mathbf{c}}_{14} \\
      \widetilde{\mathbf{c}}_{23} & \widetilde{\mathbf{c}}_{24} \\
      \widetilde{\mathbf{c}}_{33} & \widetilde{\mathbf{c}}_{34} \\
      \widetilde{\mathbf{c}}_{43} & \widetilde{\mathbf{c}}_{44} \\
      \widetilde{\mathbf{c}}_{53} & \widetilde{\mathbf{c}}_{54} \\
      \widetilde{\mathbf{c}}_{63} & \widetilde{\mathbf{c}}_{64} \\
    \end{array}
  \right)=\frac{1}{\varsigma}\left(
    \begin{array}{cc}
        \frac{-\lambda^{-1}zr^2}{2} & \lambda^{-1}z\\
        \lambda^{-1}z & -\frac{ \lambda^{-1}zr^2}{2 }  \\
   1+\frac{r^4}{4} & -r^2 \\
  -r^2 & 1+\frac{r^4}{4} \\
      \lambda\bar{z}&  -\frac{\lambda\bar{z}r^2}{2}  \\
  -\frac{\lambda\bar{z}r^2}{2} &\lambda\bar{z}\\
    \end{array}
  \right).
\]
Substituting these data into \eqref{eq-YandY}, one derives \eqref{eq-Y-ex-2}.
The rest are straightforward computations, which we will leave to interested readers.\\

{\small{\bf Acknowledgements}\ \ The author was supported by the NSFC Project No. 11571255 and the Fundamental Research Funds for the Central Universities.}

{\small
\def\refname{Reference}

}
{\small\ \

Peng Wang

Department of Mathematics, Tongji University,

Siping Road 1239, Shanghai, 200092,  P. R. China.

{\em E-mail address}: {netwangpeng@tongji.edu.cn}
 }

\end{document}